\documentclass[12pt,reqno]{amsart}
\usepackage{hyperref}
\usepackage{amsmath,amssymb}
\usepackage{caption}
\usepackage{graphicx}
\usepackage{float}
\usepackage{pgf,tikz}
\usepackage{amsmath,amscd}
\usepackage{tikz-cd}
\usepackage{mathtools}
\usepackage{mathrsfs}

\newcommand \reg{\operatorname{reg}}

\newcommand \tor{\operatorname{Tor}}

\newcommand \h{\operatorname{ht}}
\newcommand \pdim{\operatorname{pd}}

\newcommand{\Hred}{\widetilde{H}}
\newcommand \G{\mathcal{G}}
\newcommand \ee{\mathbf{e}}
\newcommand \aaa{\mathbf{a}}
\newcommand \kk{\mathbb{K}}
\newcommand \ZZ{\mathbb{Z}}
\newcommand \NN{\mathbb{N}}

\newtheorem{theorem}{Theorem}[section]
\newtheorem{definition}[theorem]{Definition}
\newtheorem{lemma}[theorem]{Lemma}
\newtheorem{proposition}[theorem]{Proposition}
\newtheorem{example}[theorem]{Example}

\newtheorem{question}[theorem]{Question}
\newtheorem{remark}[theorem]{Remark}
\newtheorem{corollary}[theorem]{Corollary}

\newtheorem{notation}[theorem]{Notation}
\newtheorem{con}[theorem]{Conjecture}

\begin{document}
	\title[Subadditivity, strand connectivity and multigraded Betti
	numbers]{Subadditivity, strand connectivity and multigraded Betti
	numbers of monomial ideals}
	\author{A V Jayanthan}
	\email{jayanav@iitm.ac.in}
	
	\author{Arvind Kumar}
	\email{arvkumar11@gmail.com}
	\address{Department of Mathematics, 5th floor, New academic block, Indian Institute of Technology
		Madras, Chennai, INDIA - 600036}
	
	\maketitle
	
\begin{abstract}
Let $R = \kk[x_1, \ldots, x_n]$ and $I \subset R$ be a homogeneous
ideal. In this article, we first obtain certain sufficient conditions
for the subadditivity of $R/I$. As a consequence, we prove that if $I$
is generated by homogeneous complete intersection, then subadditivity 
holds for $R/I$. We then study a conjecture of Avramov, Conca and 
Iyengar on subadditivity, when $I$ is a monomial ideal with $R/I$ 
Koszul. We identify several classes of edge ideals of graphs $G$ 
such that the subadditivity holds for $R/I(G)$. We then study the
strand connectivity of edge ideals and obtain several
classes of graphs whose edge ideals are strand connected. Finally,
we compute upper bounds for multigraded Betti numbers of several
classes of edge ideals.
\end{abstract}
	
\section{Introduction}
Let $R=\mathbb{K}[x_1,\ldots,x_n]$ be a standard graded polynomial
ring over a field $\mathbb{K}$.  Let $M$ be a finitely generated
graded $R$-module. Let
\[
(\mathbf{F_\bullet}, \partial_{\bullet}) : 0 \to F_p
\xrightarrow{\partial_p} \cdots \xrightarrow{\partial_2} F_1
\xrightarrow{\partial_1} F_0 \to M \to 0
\]
be a graded free resolution of $M$ (not necessarily minimal) with
$F_i = \oplus_j R(-j)^{b_{ij}}$ for some 
$b_{ij} \in \ZZ_{\geq 0}$. For $i \geq 0$, set
$$t_i^R(\mathbf{F_\bullet}) = \sup\{j ~: ~ b_{ij} \neq 0\}.$$ We say {\it subadditivity holds} for 
$(\mathbf{F_\bullet},\partial_{\bullet})$ if  for all $a,b \geq 0$ with $a+b \leq p$, 
$$t_{a+b}^R(\mathbf{F_\bullet}) \leq t_a^R(\mathbf{F_\bullet}) +
t_b^R(\mathbf{F_\bullet}).$$ 
If $(\mathbf{F_\bullet},\partial_{\bullet})$ is the graded
minimal free resolution of $M$, then write $F_i = \oplus_j
R(-j)^{\beta_{i,j}^R(M)}$, where the number $\beta_{i,j}^R(M)$ is
called the $(i,j)$-th {\it graded Betti number} of $M$, and in this
case we write $t_i^R(M)$ for $t_i^R(\mathbf{F}_\bullet)$.
	
It is known that the subadditivity does not always hold for
homogeneous ideals in polynomial rings.  Eisenbud, Huneke and Ulrich
gave an example of an ideal $I$ in a polynomial ring $R$ for which
$t_2^R(R/I) > 2 t_1^R(R/I)$, \cite[Example 4.4]{EHU}. 
%Subadditivity
%has been proved for several classes of ideals under various
%assumptions. 
In the same paper they proved that if $\dim (R/I)
\leq 1$, then $t_n^R(R/I) \leq t_a^R(R/I) +t_{n-a}^R(R/I)$, for all $a
\geq 1$. McCullough proved without any restriction on dimension,
\cite{mcl11}, that  $t_p^R(R/I) \leq \max \{t_i^R(R/I) +
t_{p-i}^R(R/I) \; : \; 1 \leq i \leq p-1\}$, where $p = \pdim ( R/I)$.
Herzog and Srinivasan, \cite{HS}, improved this result further and
showed that $t_p^R(R/I) \leq t_{p-1}^R(R/I) + t_1^R(R/I)$. They, and
independently Yazdan Pour \cite{P17}, proved that if $I$ is a
monomial ideal, then  $t_a^R(R/I) \leq t_{a-1}^R(R/I)
+ t_1^R(R/I)$ for all $a \geq 1$.  Recently,
McCullough and Seceleanu proved that the subadditivity holds for
quotients of complete intersection ideals, \cite[Proposition
4.1]{MS20}.  They also gave a family of Gorenstein cyclic $R$-modules
for which subadditivity does not hold.  For some recent developments
in this direction, see \cite{AA20, SF, FS20}.  
In \cite{ACI}, Avramov, Conca
and Iyengar also gave an example to show that the subadditivity does
not hold in general. They conjectured:
\begin{con}\cite[Conjecture 6.4]{ACI}
If $a, b \geq 1$ such that $a+b \leq \pdim(R/I)$ and
$\reg_{a+b+1}^{R/I}(\kk) = 0$, then $t_{a+b}^R(R/I) \leq t_a^R(R/I) +
t_b^R(R/I)$.
\end{con}
In the above conjecture, $\reg_a^{R/I}(\kk) =
\underset{i\leq a}{\sup}\{j - i ~ : ~\tor_i^{R/I}(\kk,\kk)_j \neq 0 \}.$
They proved that if  $\reg_{a+b+1}^{R/I}(\kk) = 0$ and $a+b \choose b$
is invertible in $\mathbb{K}$, then for $a,b  \leq \h(I)$ with $a+b
\leq \pdim ( R/I)$,  $t_{a+b}^R(R/I) \leq t_a^R(R/I) +
t_b^R(R/I)+1$. An important instance for
$\reg_{a+b+1}^{R/I}(\kk) = 0$ to happen is when $R/I$ is a Koszul
algebra. In this article, we study the above conjecture
for monomial ideals $I$ such that $R/I$ is Koszul.
It is well known that if $I$ is a monomial ideal, then
$R/I$ is Koszul if and only if $I$ is a quadratic monomial ideal.
Any quadratic monomial ideal can be polarized to
get a squarefree quadratic monomial ideal having same Betti numbers.
Thus, to study the subadditivity problem of Koszul algebras which are
quotients of monomial ideals, it is enough to study the problem for
quadratic squarefree monomial ideal. Note that quadratic
squarefree monomial ideals are in one to one correspondence with
finite simple graphs. This gives us the extra leverage of using
combinatorial tools to study algebraic invariants of these ideals. 
Let $G$ be a finite simple graph
on $V(G)=\{x_1,\ldots,x_n\}$. Then the {\it edge ideal} of $G$,
denoted by $I(G)$, is the ideal generated by the set $\{ x_ix_j \; :
\: \{x_i,x_j\} \in E(G) \}$. Abedelfatah and Nevo, \cite{AN}, proved
that for any graph $G$ on $V(G)$ over any field $\mathbb{K}$,
$t_a^R(R/I(G)) \leq t_{a-i}^R(R/I(G)) + t_i^R(R/I(G))$ for all $a \geq
1$ and $i=1,2,3$.  Bigdeli and Herzog, \cite{BH17}, showed the
subadditivity holds for edge ideals of chordal graphs and whisker
graphs.

In this article, we prove that the subadditivity holds for several
classes of edge ideals. First we study the problem for homogeneous
ideals with some extra hypotheses, Theorems \ref{subad}, \ref{rsub},
\ref{tsub}. As a consequence of our results, we reprove a result of
McCullough and Seceleanu that subadditivity holds for homogeneous
complete intersections, Corollary \ref{csub}. We then move on to study
the subadditivity problem for edge ideals of finite simple graphs. We
first show that if $G$ is a graph and $I(G)$ its edge ideal, then
$t_{a+b}^R(R/I(G)) \leq t_a^R(R/I(G)) + t_b^R(R/I(G))$ if $a \leq
\nu(G)+1$, Propositions \ref{indsub}, \ref{nug+1}, where $\nu(G)$
denotes the induced matching number (see
Section 2 for the definition). We then study the
(multi)graded Betti numbers of a hereditary class of graphs under some
hypotheses and show that the subadditivity holds for the edge ideals
of this class. As a consequence, we prove that the subadditivity holds
for clique sum of a cycle and chordal graphs (in particular, unicyclic
graphs), Wheel graphs, Jahangir graphs, complete multipartite graphs
and fan graphs, Theorem \ref{subedge}, Corollary \ref{unicyc},
Corollary \ref{main-cor}. Our methods give new ways of constructing
more classes of graphs having subadditivity. We also consider $t$-path
ideals, which is a generalization of edge ideals (or can be thought of
as certain $t$-uniform hypergraphs) of rooted trees and prove that the
subadditivity holds for these ideals, Theorem \ref{rooted-tree}.

A closely related problem on vanishing of Betti numbers is the strand
connectivity. For a finitely generated graded $R$-module $M$, the set
$\{i : \beta_{i, i+j}^R(M) \neq 0 \}$ is called the $j$-$strand$ of $M$. 
If $j$-strand of $M$ is non-empty, then set 
$$p_j(M):=\max\{i \; : \; \beta_{i,i+j}^R(M) \neq
0\}\hspace*{5mm}\text{ and }\hspace*{5mm}
q_j(M):=\min\{i \; : \; \beta_{i,i+j}^R(M) \neq 0\}.$$
A non-empty $j$-strand of $M$ is said to be {\it connected} if $j$-strand
$=[q_j(M),p_j(M)]$. The module $M$ is said to be {\it strand connected} if
every non-empty strand of $M$ is connected. For a homogeneous ideal
$I \subset R$, we set these terminologies for $I$ by taking $M =
R/I$.

What are strand connected homogeneous ideals? Well, there are some
obvious classes in this category. For example, if $I$ has a pure
resolution, then the non-empty strands are always connected. It is not
very difficult to see that the strands of not all monomial ideals are
connected, see Example \ref{discon-strand}. So, one is interested in
identifying classes of monomial ideals which are strand connected.
Even in the case of quadratic monomial ideals, not many classes of
ideals have been identified which are strand connected. In this
context, Conca asked:

\begin{question}\cite[Question 1.1]{AN}
If $I$ is a quadratic monomial ideal, then is $I$ strand connected?
\end{question}

In \cite{AN}, Abedelfatah and Nevo gave a class of quadratic monomial
ideals that are not strand connected. They also proved that the
$2$-strand of $I$ is connected for any quadratic monomial ideal. In
\cite{BH17}, Bigdeli and Herzog proved that edge ideals of chordal
graphs and cycles are strand connected. Our goal is to identify more
classes of edge ideals which are strand connected.

We begin with an example of a monomial ideal which is not strand
connected. We then identify a hereditary class of graphs whose edge
ideals are strand connected, Theorem \ref{subedge}. Then we prove that
from a given edge ideal which is strand connected, one can obtain
stand connected edge ideals by doing certain combinatorial operations
on it, Theorems \ref{strandvertex}, \ref{strandjoin}. As a
consequence, we prove strand connectivity for several important
classes of graphs, Corollary \ref{strand-cor}.

Most of the important homological invariants associated with finitely
generated modules are read off from the Betti numbers. Graded Betti
numbers of edge ideals of some classes of graphs are known, see for
example \cite{HV07}, \cite{JTh}, \cite{Kat2006}. The structure of the
minimal free resolution can be further refined by considering
multigraded resolution and multigraded Betti numbers. If $G$ is a
forest on $n$ vertices, then Bouchat, in \cite{BRR}, proved that the
$\NN^n$-graded Betti numbers of $I(G)$ are either $0$ or $1$. Boocher
et al. showed that the $\NN^n$-graded Betti numbers of cycles are
bounded above by $2$. In this article, we generalize this result to
the case of unicyclic graphs and show that the $\NN^n$-graded Betti
numbers are bounded above by $2$, Theorem \ref{multi-betti-uc}. We
also obtain upper bounds for the $\NN^n$-graded Betti numbers of Fan
graphs, Jahangir graphs and complete multipartite graphs, Corollaries
\ref{multi-betti-fan}, \ref{multi-betti-jh}, \ref{multi-betti-cm}.
The paper is organized as follows: In Section 2, we prove the results
concerning the subadditivity of monomial ideals. We collect all the
results on strand connectivity in the next section and final section
contains the results on multigraded Betti numbers.

\section{Subadditivity of syzygies of homogeneous ideals}
In this section, we study the subadditivity of maximal shifts in the
finite graded free resolution of homogeneous ideals. We first discuss
certain sufficient conditions for the subadditivity of monomial ideals.
	
\begin{theorem}\label{subad}
Let $I \subset R$ be a homogeneous ideal, and $f \in R$ be a
homogeneous polynomial of degree $d>0$ such that $f\notin I$.
Assume that $t_1^R(R/(I:f)) \leq d$. If free
resolutions  of $R/I$ and $R/(I:f)$ satisfy the subadditivity
condition, then the resolution obtained by the mapping cone
construction applied to the map $[R/(I:f )](-d) \xrightarrow{\cdot
f} R/I$ satisfies the subadditivity condition.
\end{theorem}
\begin{proof}
Let $(\mathbf{F}_{\bullet}, \delta^{\mathbf{F}}_{\bullet})$ and
$(\mathbf{G}_{\bullet},\delta^{\mathbf{G}}_{\bullet})$
be free resolutions of $R/I$ and $R/(I:f)$, respectively. Note
that $(\mathbf{G}_{\bullet}(-d),\delta^{\mathbf{G}}_{\bullet})$ is a free resolution
of $[R/(I:f)](-d)$. Let $(\mathfrak{F}_{\bullet},\xi_{\bullet})$
be the mapping cone construction applied to the map $[R/(I:f
)](-d) \xrightarrow{\cdot f} R/I$. Then $\mathfrak{F}_i = F_i
\oplus G_{i-1}(-d)$, for all $i \geq 1$. Therefore, for all $a,b \geq 1$.
\begin{align*}
t^R_{a}(\mathfrak{F}_{\bullet})+t^R_{b}(\mathfrak{F}_{\bullet})&=
\max\{t^R_{a}(\mathbf{F}_{\bullet}),\; t^R_{a-1}(\mathbf{G}_{\bullet}(-d))\}+\max\{t^R_{b}(\mathbf{F}_{\bullet}),\; t^R_{b-1}(\mathbf{G}_{\bullet}(-d))\}\\&=
\max\{t^R_{a}(\mathbf{F}_{\bullet}),\; t^R_{a-1}(\mathbf{G}_{\bullet})+d\}+\max\{t^R_{b}(\mathbf{F}_{\bullet}),\; t^R_{b-1}(\mathbf{G}_{\bullet})+d\}\\&=
\max \{t^R_{a}(\mathbf{F}_{\bullet})+t^R_{b}(\mathbf{F}_{\bullet}),\; t^R_{a}(\mathbf{F}_{\bullet})+t^R_{b-1}(\mathbf{G}_{\bullet})+d,\;\\& \;\;\;\;\;\;\;\;\;\;\;\;\;\;\; t^R_{a-1}(\mathbf{G}_{\bullet})+t^R_{b}(\mathbf{F}_{\bullet})+d,\; t^R_{a-1}(\mathbf{G}_{\bullet})+t^R_{b-1}(\mathbf{G}_{\bullet})+2d\}\\&\geq
\max \{t^R_{a}(\mathbf{F}_{\bullet})+t^R_{b}(\mathbf{F}_{\bullet}),\; t^R_{a-1}(\mathbf{G}_{\bullet})+t^R_{b-1}(\mathbf{G}_{\bullet})+2d\}\\&=
\max \{t^R_{a}(\mathbf{F}_{\bullet})+t^R_{b}(\mathbf{F}_{\bullet}),\; t^R_{a-1}(\mathbf{G}_{\bullet})+d+t^R_{b-1}(\mathbf{G}_{\bullet})+d\}\\&\geq
\max \{t^R_{a}(\mathbf{F}_{\bullet})+t^R_{b}(\mathbf{F}_{\bullet}),\; t^R_{a-1}(\mathbf{G}_{\bullet})+t^R_1(\mathbf{G}_{\bullet})+t^R_{b-1}(\mathbf{G}_{\bullet})+d\}\\&\geq
\max \{t^R_{a}(\mathbf{F}_{\bullet})+t^R_{b}(\mathbf{F}_{\bullet}),\; t^R_{a}(\mathbf{G}_{\bullet})+t^R_{b-1}(\mathbf{G}_{\bullet})+d\}\\&\geq 
\max\{t^R_{a+b}(\mathbf{F}_{\bullet}),\; t^R_{a+b-1}(\mathbf{G}_{\bullet})+d\}\\&=
\max\{t^R_{a+b}(\mathbf{F}_{\bullet}),\; t^R_{a+b-1}(\mathbf{G}_{\bullet}(-d))\}\\&=
t^R_{a+b}(\mathfrak{F}_{\bullet}).
\end{align*}
		Hence, the assertion follows.
	\end{proof}
	
	In Theorem \ref{subad}, if we drop the condition $t_1^R(R/(I:f)) \leq d$, then the result need no be true. We illustrate this by the following example.
	\begin{example}{\em
Let $R = 
\mathbb{K}[x_1,x_2,x_3,x_4,x_5,x_6,x_7,y_1,y_2,y_3,y_4,y_5,y_6]$.
Let $f=x_1y_6x_7-x_6y_1x_7$ and $I=(x_1y_2-x_2y_1,x_2y_3-x_3y_2,
x_3y_4-x_4y_3,x_4y_5-x_5y_4,x_5y_6-x_6y_5).$
Since $x_7$ is regular modulo $I$, $I:f =I:(x_1y_6-x_6y_1)$. Therefore, it follows from
\cite[Theorem 2.4]{Leila18} that the mapping cone applied to $0
\to [R/(I:f)](-3) \xrightarrow{\cdot f} R/I$ gives the minimal
free resolution of $R/(I,f)$. Using Macaulay2 \cite{M2}, we get
\\[1ex]
\noindent
\begin{minipage}{\linewidth}
	\begin{minipage}{0.45\linewidth}
		\begin{table}[H]
			\centering
			\caption{Betti diagram of $R/I$}
			\scalebox{1}{%
				\begin{tabular}{c|cccccc}
					&0&1&2&3&4&5\\
					\hline 
					\text{0:}&1&\text{.}&\text{.}&\text{.}&\text{.}&\text{.}\\\text{1:}&\text{.}&5&\text{.}&\text{.}&\text{.}&\text{.}\\\text{2:}&\text{.}&\text{.}&10&\text{.}&\text{.}&\text{.}\\
								\text{3:}&\text{.}&\text{.}&\text{.}&10&\text{.}&\text{.}\\
								\text{4:}&\text{.}&\text{.}&\text{.}&\text{.}&5&\text{.}\\
								\text{5:}&\text{.}&\text{.}&\text{.}&\text{.}&\text{.}&1\\
						\end{tabular}}\label{betti1}
					\end{table} 
				\end{minipage}
				\begin{minipage}{0.52\linewidth}
					\begin{table}[H]
						\centering
						\caption{Betti diagram of $R/(I:f)$}
						\scalebox{1}{%
							\begin{tabular}{c|cccccc}
								&0&1&2&3&4&5\\
								\text{0:}&1&\text{.}&\text{.}&\text{.}&\text{.}&\text{.}\\
								\text{1:}&\text{.}&5&\text{.}&\text{.}&\text{.}&\text{.}\\
								\text{2:}&\text{.}&\text{.}&10&\text{.}&\text{.}&\text{.}\\
								\text{3:}&\text{.}&5&24&55&40&10\\	
						\end{tabular}}\label{betti2}
					\end{table}
				\end{minipage}
			\end{minipage}
			\begin{table}[H]
				\centering
				\caption{Betti diagram of $R/(I,f)$}
				\scalebox{1}{%
					\begin{tabular}{c|ccccccc}
						&0&1&2&3&4&5&6\\
						\hline 
						\text{0:}&1&\text{.}&\text{.}&\text{.}&\text{.}&\text{.}&\text{.}\\
						\text{1:}&\text{.}&5&\text{.}&\text{.}&\text{.}&\text{.}&\text{.}\\
						\text{2:}&\text{.}&1&10&\text{.}&\text{.}&\text{.}&\text{.}\\
						\text{3:}&\text{.}&\text{.}&5&10&\text{.}&\text{.}&\text{.}\\
						\text{4:}&\text{.}&\text{.}&\text{.}&10&5&\text{.}&\text{.}\\
						\text{5:}&\text{.}&\text{.}&5&24&55&41&10\\
				\end{tabular}}\label{betti3}
			\end{table}
			It follows from Tables \ref{betti1} and \ref{betti2} that
			subadditivity holds for $R/I$ and $R/(I:f)$. Also, $t_1^R(R/(I:f))
			=4>3$. Observe from Table \ref{betti3} that $t_1^R(R/(I,f)) =3$ and $t_2^R(R/(I,f))
			=7>2t_1^R(R/(I,f))$. Thus, subadditivity does not hold for
			$R/(I,f)$. } \qed
	\end{example}
	In the following result, we prove that we can drop the condition $t_1^R(R/(I:f)) \leq d$, if $I:f=I$.
	\begin{theorem}\label{rsub}
		Let $I \subset R$ be a homogeneous ideal, and $f \in R$ be a
		homogeneous polynomial of degree $d>0$ such that $ I:f=I$. If subadditivity holds for
		$R/I$, then the subadditivity holds for $R/(I,f)$.
	\end{theorem}
	\begin{proof}
		Let $(\mathbf{F}_{\bullet}, \delta^{\mathbf{F}}_{\bullet})$ 
		be the minimal free resolution of $R/I$. Note
		that $(\mathbf{F}_{\bullet}(-d),\delta^{\mathbf{F}}_{\bullet})$ is the minimal free resolution
		of $[R/I](-d)$. 	Since $I:f=I$, the mapping cone construction applied to the map $[R/(I:f)](-d)  \xrightarrow{\cdot f} R/I$ gives the minimal free resolution of $R/(I,f)$. Therefore, for each $ 1 \leq i \leq \text{pd}(R/(I,f))$, $$t_i^R\left(\frac{R}{(I,f)}\right) = \max\Big\{t_i^R\left(\frac{R}{I}\right),t_{i-1}^R\left(\frac{R}{I}(-d)\right)\Big\}=\max\Big\{t_i^R\left(\frac{R}{I}\right),t_{i-1}^R\left(\frac{R}{I}\right)+d\Big\}.$$ Thus, for all $a,b \geq
		1$ with $a+b \leq \text{pd}(R/(I,f))$, 
		\begin{align*}
			t^R_{a+b}\left(\frac{R}{(I,f)}\right)&= 
			\max\Big\{t^R_{a+b}\left(\frac{R}{I}\right),\; t^R_{a+b-1}\left(\frac{R}{I}\right)+d\Big\}\\& \leq  \max \Big\{t^R_{a}\left(\frac{R}{I}\right)+t^R_{b}\left(\frac{R}{I}\right),\; t^R_{a}\left(\frac{R}{I}\right)+t^R_{b-1}\left(\frac{R}{I}\right)+d\Big\}\\& \leq
			\max\Big\{t^R_{a}\left(\frac{R}{I}\right),\; t^R_{a-1}\left(\frac{R}{I}\right)+d\Big\}+\max\Big\{t^R_{b}\left(\frac{R}{I}\right),\; t^R_{b-1}\left(\frac{R}{I}\right)+d\Big\}\\&=
			t^R_{a}\left(\frac{R}{(I,f)}\right)+t^R_{b}\left(\frac{R}{(I,f)}\right)			
		\end{align*}
		Hence, the assertion follows.	
	\end{proof}
	
	As an immediate consequence, we derive a result of McCullough and
	Seceleanu:
	\begin{corollary}\cite[Proposition 4.1]{MS20}\label{csub}
		If $I$ is a homogeneous complete intersection, then
		subadditivity holds for $R/I$.
	\end{corollary}
	The following result says that if $I$ and $J$ are homogeneous ideals
	in distinct polynomial rings over same field and subadditivity holds
	for these two ideals, then subadditivity holds for the ideal generated
	by their sum in the tensor product of these two polynomial rings.
	\begin{theorem}\label{tsub}
		Let $I,J$  be homogeneous ideals of $R$ such that there exist minimal
		generating sets for $I$ and $J$ in disjoint sets of variables. If
		subadditivity holds for $R/I$ and $R/J$, then subadditivity holds
		for $R/(I+J)$.
	\end{theorem}
	\begin{proof}
		Since the minimal generating sets of $I$ and $J$ are in disjoint
		variables,  the tensor product of the minimal free resolution of
		$R/I$ and $R/J$ provides the minimal free resolution of $R/(I+J)$.
		In particular, 
		$$\beta_{i,j}^R\left(\frac{R}{I+J}\right) = \sum_{0 \leq r\leq i,\;r
			\leq s \leq j}
		\beta_{r,s}^R\left(\frac{R}{I}\right)\beta_{i-r,j-s}^R\left(\frac{R}{J}\right).$$
		Thus, it is straightforward to verify that for each $i \geq 1$, 
		$$t_i^R\left(\frac{R}{I+J}\right)=\max_{0\leq r\leq i}\Big
		\{t_r^R\left(\frac{R}{I}\right)+t_{i-r}^R\left(\frac{R}{J}\right)\Big\}.$$
		Let $a,b \geq 1$ such that $a+b \leq \text{pd}(R/(I+J))$. Then, we
		have 
		\begin{align*}
			t_{a+b}^R\left(\frac{R}{I+J}\right) & =\max_{0 \leq r \leq a+b}\Big
			\{t_r^R\left(\frac{R}{I}\right)+t_{a+b-r}^R\left(\frac{R}{J}\right)\Big\}
			\\ &=\max_{0 \leq i\leq a, \;0 \leq k \leq b}\Big
			\{t_{{i+k}}^R\left(\frac{R}{I}\right)+t_{a+b-i-k}^R\left(\frac{R}{J}\right)\Big\}
			\\ &\leq \max_{0 \leq i \leq a, \; 0 \leq k \leq b}\Big
			\{t_i^R\left(\frac{R}{I}\right)+t_k^R\left(\frac{R}{I}\right)+t_{a-i}^R\left(\frac{R}{J}\right)+t_{b-k}^R\left(\frac{R}{J}\right)\Big\}
			\\ & \leq \max_{0 \leq i \leq a}\Big
			\{t_i^R\left(\frac{R}{I}\right)+t_{a-i}^R\left(\frac{R}{J}\right)\Big\}+\max_{0\leq k\leq b}\Big\{t_k^R\left(\frac{R}{I}\right)+t_{b-k}^R\left(\frac{R}{J}\right)\Big\}\\&
			=t_{a}^R\left(\frac{R}{I+J}\right)+t_{b}^R\left(\frac{R}{I+J}\right).
		\end{align*}
		\iffalse 
		\begin{align*}
			t_{a+b}^R\left(\frac{R}{I+J}\right) & =\max_{r+r'=a+b}\Big
			\{t_r^R\left(\frac{R}{I}\right)+t_{r'}^R\left(\frac{R}{J}\right)\Big\}
			\\ &=\max_{i+j=a, \;k+l=b}\Big
			\{t_{{i+k}}^R\left(\frac{R}{I}\right)+t_{j+l}^R\left(\frac{R}{J}\right)\Big\}
			\\ &\leq \max_{i+j=a, \; k+l=b}\Big
			\{t_i^R\left(\frac{R}{I}\right)+t_k^R\left(\frac{R}{I}\right)+t_j^R\left(\frac{R}{J}\right)+t_l^R\left(\frac{R}{J}\right)\Big\}
			\\ & \leq \max_{i+j=a}\Big
			\{t_i^R\left(\frac{R}{I}\right)+t_j^R\left(\frac{R}{J}\right)\Big\}+\max_{k+l=b}\Big\{t_k^R\left(\frac{R}{I}\right)+t_l^R\left(\frac{R}{J}\right)\Big\}\\&
			=t_{a}^R\left(\frac{R}{I+J}\right)+t_{b}^R\left(\frac{R}{I+J}\right).
		\end{align*} \fi
		This completes the proof.
	\end{proof}
	\subsection{Edge ideals of graphs} In this subsection, we study the subadditivity problem for quadratic squarefree
	monomial ideals. First, we recall some notion from graph theory.
	
	Let $G$  be a  finite simple graph with vertex set $V(G)$ and edge set
	$E(G)$. If $E(G)=\emptyset$, then we say that $G$ is a {\it trivial
		(or empty) graph}. For $A \subseteq V(G)$, $G[A]$ denotes the
	\textit{induced subgraph} of $G$ on the vertex set $A$, i.e., for $i,
	j \in A$, $\{i,j\} \in E(G[A])$ if and only if $ \{i,j\} \in E(G)$.
	For $A \subset V(G)$, $G \setminus A$ denotes the  induced subgraph of
	$G$ on the vertex set $V(G) \setminus A$.  The {\it neighborhood of a
		vertex} $v$, denoted by $N_G(v),$ is defined as $\{u \in V(G) :
	\{u,v\} \in E(G)\}$. We set $N_G[v]=N_G(v) \cup \{v\}$. The {\it degree of
		a vertex} $v$ is  $|N_G(v)|$, and it is denoted by $\deg_G(v)$. If $\deg_G(v)=1$, then we say that $v$ is a {\it pendant
		vertex}. For $e \in E(G)$, $G\setminus e$ is the graph on the vertex
	set $V(G)$ and edge set $E(G) \setminus \{e\}$.  
	
	A connected graph $G$ is said to be a {\it cycle} if $\deg_G(v)=2$, for all $v \in V(G)$.  A cycle $G$ is a {\it $k$-cycle}
	if $|V(G)|=k$, and it is denoted by $C_k$. A  \textit{tree} is a
	connected graph $G$ such that $k$-cycle is not an induced subgraph of $G$,  for
	all $k \geq 3$.  A \textit{path} is a tree which has exactly two
	pendant vertices.  We say that
	$G$ is a {\it chordal graph} if  $k$-cycle is not an induced subgraph
	of $G$,  for all $k \geq 4$.

	A set of pairwise disjoint edges in a graph G is called a {\it
		matching}.  If a matching is an induced subgraph, then such matching
	is called an {\it induced matching}. The largest size of an induced
	matching in $G$ is called {\it induced matching number} of $G$, and it
	is denoted by $\nu(G).$  A subset $C \subset V(G)$ is said to be a
	{\it vertex cover} of $G$ if for each $e \in  E(G)$, $e \cap C \neq
	\emptyset$. If $C$ is minimal with respect to inclusion, then $C$ is
	called a {\it minimal vertex cover} of $G$.

	Below, we fix some notation for the rest of the paper. 
	\begin{notation}
		If $G$ is a graph on $n$ vertices, then we set
		$V(G)=\{x_1,\ldots,x_n\}$ and the edge ideal $I(G) = (x_ix_j ~ : ~
		\{x_i, x_j\} \in E(G)\})$ to be an ideal of
		$R=\mathbb{K}[x_1,\ldots,x_n]$. Also, we set $t_a(G) :=
		t_a^R(R/I(G))$. For a graph $G$, by subadditivity of $G$
		we mean subadditivity of $R/I(G)$.
	\end{notation}
	
	We now begin the study of the subadditivity problem for edge ideals of
	graphs.
	
	\begin{proposition}\label{indsub}
		For $b \leq \nu(G)$, and $a \geq 1$ with $a+b \leq \pdim(R/I(G))$,
		$$t_{a+b}(G)\leq t_{a}(G)+t_b(G).$$
	\end{proposition}
	\begin{proof}
		Let $b \leq \nu(G)$, and $a \geq 1$ be such that $a+b \leq
		\pdim(R/I(G))$. It follows from \cite[Lemma 2.2]{Kat2006},
		$t_b(G)=2b=b \cdot t_1(G)$. Now, apply \cite[Corollary 4]{HS}
		$b$ times, we get $t_{a+b}(G) \leq t_a(G)+b \cdot t_1(G)$, which completes the proof.
	\end{proof}
	It follows from \cite[Lemma 2.2]{Kat2006} that $t_a(G) < 2a$ for all $a >
	\nu(G)$. In the following auxiliary lemma, we compute $t_a(G)$ for $a =
	\nu(G) + 1$.
	%\textcolor{red}{Now, we prove an auxillary result.}
	\begin{lemma}\label{tb}
		If $G$ is not a disjoint union of edges, then 
		$t_{\nu(G)+1}(G) = 2\nu(G) + 1$.
	\end{lemma}
	\begin{proof}
		Set $b = \nu(G) + 1$. It is easy to see that if $C$ is a minimal
		vertex cover of $G$, then $\nu(G) \leq |C|$. Therefore, $\nu(G)
		\leq \h(I(G))$.
		By \cite[Lemma 2.2]{Kat2006}, $t_{b-1}(G)=2b-2$ and
		$t_b(G) <2b$. If $\nu(G) < \h(I(G))$, then it follows from
		\cite[Lemma 6.1]{ACI} that $t_{b-1}(G)<t_b(G)$.
		Thus, $t_b(G)=2b-1 = 2\nu(G) + 1$. Suppose $\nu(G) = \h(I(G))$.
		Let $\{e_1, \ldots, e_{\nu(G)}\}$ be an induced matching in
		$G$. Since $G$ is not a disjoint union of edges, there is an edge $f$
		such that $f \cap e_i \neq \emptyset$ for some $1 \leq i \leq \nu(G)$.
		Set $e_i = \{x_i, y_i\}$. Without loss of generality, we may assume
		that $\{x_1, \ldots, x_{\nu(G)}\}$ is a vertex cover and $f = \{x_1,
		z\}$. Since $\h(I(G)) = \nu(G)$, $\{y_i, z\} \notin E(G)$ for any
		$1 \leq i \leq \nu(G)$. Without loss of generality, we may assume that
		$N_G(z) = \{x_1, \ldots, x_k\}$. Let $H$ denote the induced subgraph
		of $G$ on the vertex set $\{x_1,\ldots, x_{\nu(G)}, y_1, \ldots,
		y_{\nu(G)}, z\}$. Then $H$ is a graph as shown in Figure \ref{fig1}.
		Let $H_1$ be the induced subgraph on $\{z, x_1, \ldots, x_k, y_1,
		\ldots, y_k\}$ and $H_2$ be the induced subgraph on $\{x_{k+1},
		\ldots, x_{\nu(G)}, y_{k+1}, \ldots, y_{\nu(G)}\}$. Since the minimal 
		free resolution of $R/I(H)$ is given by the tensor product of
		minimal free resolutions of $R/I(H_1)$ and $R/I(H_2)$, it can be seen that
		$\beta_{i,j}^R(R/I(H)) = \sum_{0 \leq r \leq s}\beta_{i-r,j-s}^R(R/I(H_1))
		\beta_{r,s}^R(R/I(H_2))$.  To prove
		the main assertion, it is enough to prove that $\beta^R_{b,
			2b-1}(R/I(H)) \neq 0$ and to prove this statement, we prove that $\beta_{k+1, 2k+1}^R(R/I(H_1)) \neq 0$ and $
		\beta_{b-1-k, 2(b-1-k)}^R(R/I(H_2)) \neq 0$.
		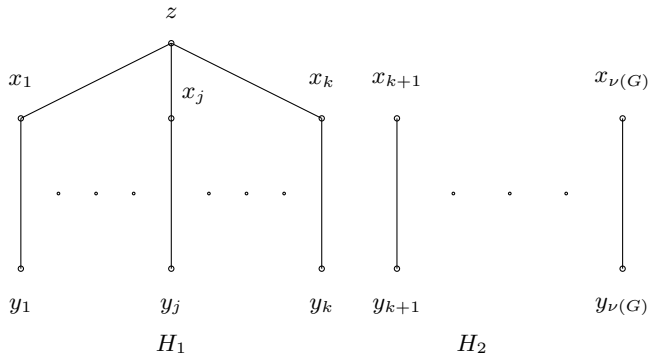
\begin{figure}[H]
			
			\begin{tikzpicture}[scale=1]
			
			\draw (-4,0)-- (-4,-2);
			
			\draw (0,0)-- (0,-2);
			
			\draw (1,0)-- (1,-2);
			
			\draw  (4,0)-- (4,-2);
			
			\draw (-2,1)-- (-4,0);
			
			\draw (-2,1)-- (0,0);
			
			\draw (-2,0)-- (-2,-2);
			
			\draw (-2,1)-- (-2,0);
			
			\begin{scriptsize}
			
			\draw  (-4,-2) circle (1pt);
			
			\draw   (-4,-2.5) node {$y_1$};
			
			\draw   (-4,0) circle (1pt);
			
			\draw   (-4,0.5) node {$x_1$};
			\draw   (-3,-1) circle (.5pt);
			\draw   (-3.5,-1) circle (.5pt);
			\draw   (-2.5,-1) circle (.5pt);
			\draw   (-1.5,-1) circle (.5pt);
			\draw   (-1,-1) circle (.5pt);
			\draw   (-0.5,-1) circle (.5pt);
			
			\draw   (0,0) circle (1pt);
			
			\draw   (0,0.5) node {$x_k$};
			
			\draw   (0,-2) circle (1pt);
			
			\draw   (0,-2.5) node {$y_k$};
			
			\draw   (1.75,-1) circle (.5pt);
			\draw   (2.5,-1) circle (.5pt);
			\draw   (3.25,-1) circle (.5pt);
			
			\draw   (1,0) circle (1pt);
			
			\draw   (1,0.5) node {$x_{k+1}$};
			
			\draw   (1,-2) circle (1pt);
			
			\draw   (1,-2.5) node {$y_{k+1}$};
			
			\draw   (4,0) circle (1pt);
			
			\draw   (4,0.5) node {$x_{\nu(G)}$};
			
			\draw   (4,-2) circle (1pt);
			
			\draw   (4,-2.5) node {$y_{\nu(G)}$};
			
			\draw   (-2,1) circle (1pt);
			
			\draw   (-2,1.4) node {$z$};

			\draw   (-2,0) circle (1pt);
			
			\draw   (-1.7,0.3) node {$x_j$};
			
			\draw   (-2,-2) circle (1pt);
			
			\draw   (-2,-2.5) node {$y_j$};
			
			\draw (-2,-3) node {$H_1$};
			\draw (2,-3) node {$H_2$};
			%\draw (0,-3.5) node {$H$};
			\end{scriptsize}
			
			\end{tikzpicture}
			\caption{$H = H_1 \sqcup H_2$}\label{fig1}
		\end{figure}
		
		Let $H'$ be the induced subgraph of $H_1$ on $\{x_1,
		\ldots, x_k, y_1, \ldots, y_k\}$ and $H''$ be the
		induced subgraph on $\{z, x_1, \ldots, x_k\}$. Then, it follows from
		\cite[Theorem 4.6]{HV07} that 
		\[
		\beta^R_{i,j}(R/I(H_1)) = \beta_{i,j}^R(R/I(H')) + \beta_{i,j}^R(R/I(H'')) +
		\beta_{i-1,j-1}(R/I(H')).
		\]
		Since $H'$ is a disjoint union of $k$ edges, $I(H')$ is a complete intersection ideal, and thus, $\beta_{k,
			2k}^R(R/I(H')) \neq 0$. Consequently, $\beta_{k+1,
			2k+1}^R(R/I(H_1)) \neq 0$. Since $H_2$ is a disjoint union of $b-1-k$
		edges, $I(H_2)$ is complete intersection,  $\beta_{b-1-k, 2(b-1-k)}^R(R/I(H_2)) \neq 0$. Hence,
		$\beta^R_{b,2b-1}(R/I(H)) \neq 0,$
		which concludes the proof.
	\end{proof}

	Let $\Delta$ be a simplicial complex on $\{x_1, \ldots, x_n\}$. For $V
	\subseteq \{x_1, \ldots, x_n\}$, the subcomplex of $\Delta$ on 
	$V$ is $\Delta[V]=\{ F \in \Delta \; : \; F \subseteq
	V \}.$ The {\it Stanly-Reisner ideal} $I_{\Delta}$ of the simplicial
	complex $\Delta$ is the ideal generated by squarefree monomials $x_F=
	\prod_{x_i \in F} x_i$ with $F \notin \Delta$, $F \subset \{x_1,
	\ldots, x_n\}$.  Let $G$
	be a graph on the vertex set $V(G)$. A subset $U \subset V(G)$ is said
	to be an {\it independent set} if $G[U]$ is a trivial graph. Let
	$\Delta_G=\{ U \subset V(G) \; : \; U \text{ is an independent set in
	} G \}$. Then $\Delta_G$ is a simplicial complex, called the {\it
		independence complex } of $G$. It is easy to observe that $I(G)
	=I_{\Delta_G}$.
	
	%\begin{theorem}(Hochster)\label{HochstersFormula}
	%Let $\Delta$ be a simplicial complex on $[n]$. Then, for all $ i \geq 1$ and $j \geq i$.
	%$$\beta_{i,j}^R(R/I_{\Delta})=\sum_{W\subset [n],\; |W|=j}\dim_{\mathbb{K}} \Hred_{j-i-1}(\Delta[W]; \mathbb{K}).$$
	%\end{theorem}

\begin{proposition}\label{nug+1}
Let $b =\nu(G)+1$. Then, for all $a \geq 1$ with
$a+b \leq \pdim(R/I(G))$, $$t_{a+b}(G)\leq
t_{a}(G)+t_b(G).$$
\end{proposition}
\begin{proof} 
Let $\Delta$ denote the independent complex of $G$. Then $I_\Delta =
I(G)$. By Hochster's fomula \cite{Hoc77}, we get 
\begin{eqnarray}\label{HochstersFormula}
	\beta_{i,j}^R(R/I_{\Delta})=\sum_{W\subset V(G),\;
		|W|=j}\dim_{\mathbb{K}} \Hred_{j-i-1}(\Delta[W]; \mathbb{K}).
\end{eqnarray}
Taking $i = a+b$ and $j = t_{a+b}(G)$, we get
$\Hred_{l}(\Delta[W]; \mathbb{K}) \neq 0$ for some $W \subset
V(G)$, where $l=t_{a+b}(G)-(a+b)-1$. Set
$\Delta'=\Delta[W]$ and $H=G[W]$. We claim that $H$ has no
isolated vertices. If $H$ has an isolated vertex, say $x$, then $x \in
F$ for every facet $F$ of $\Delta'$. Hence $\Delta'$
is a cone of a simplicial complex and hence an acyclic complex. 
Consequently,
$\Hred_{l}(\Delta'; \mathbb{K}) = 0$, which is a contradiction.
Thus, $H$ has no isolated vertices. Now, we claim that there is a
vertex of degree at least two in $H$. If not, then every vertex of
$H$ has degree one, i.e., $H$ is a disjoint union of edges.
Therefore, $I(H)$ is a complete intersection. Since
$\beta_{a+b,t_{a+b}(G)}^R(R/I(H)) \neq 0$,
$t_{a+b}(G)=2(a+b)$, which is a contradiction to the fact
that $t_i(G) < 2i $ for $i > \nu(G)$, see \cite[Lemma
2.2]{Kat2006}. Thus, $H$ has a vertex of degree at least two. 

Let $v$ be a vertex of degree at least two in $H$. Let $x,y \in N_H(v)$ such that $x \neq y$. Observe that $\Delta'=(\Delta' \setminus \{v\}) \cup (\Delta' \setminus \{x,y\})$. Set $ \Delta_1=\Delta' \setminus \{v\} $ and $\Delta_2=\Delta'\setminus \{x,y\}$. Consider the long exact sequence of reduced homologies, $$\cdots \to \Hred_{l}(\Delta_1 \cap \Delta_2; \mathbb{K}) \to \Hred_{l}(\Delta_1; \mathbb{K}) \oplus \Hred_{l}(\Delta_2; \mathbb{K}) \to \Hred_{l}(\Delta'; \mathbb{K}) \to \Hred_{l-1}(\Delta_1 \cap \Delta_2; \mathbb{K}) \to \cdots .$$ 
		If $\Hred_{l}(\Delta_1; \mathbb{K}) \neq 0$, then by
		(\ref{HochstersFormula}),  $\beta_{a+b-1,t_{a+b}(G)-1}^R(R/I(G))
		\neq 0$. Consequently, $t_{a+b}(G) \leq t_{a+b-1}(G)+1$. Now, by Proposition \ref{indsub}, we have 
		\[
		t_{a+b}(G) \leq
		t_{a}(G)+t_{b-1}(G)+1=t_{a}(G)+2(b-1)+1=t_{a}(G)+t_{b}(G).
		\]
		If $\Hred_{l}(\Delta_1; \mathbb{K}) = 0$, then $\Hred_{l}(\Delta_2;
		\mathbb{K}) \neq 0$ or $\Hred_{l-1}(\Delta_1\cap \Delta_2;
		\mathbb{K}) \neq 0$. Assume that $\Hred_{l}(\Delta_2; \mathbb{K})
		\neq 0$. Then, by (\ref{HochstersFormula}),
		$\beta_{a+b-2,t_{a+b}(G)-2}^R(R/I(G)) \neq 0$ which further
		implies that $t_{a+b}(G) \leq t_{a+b-2}(G)+2$. Again, by Proposition \ref{indsub}, we get 
		\[
		t_{a+b}(G) \leq
		t_{a}(G)+t_{b-2}(G)+2=t_{a}(G)+2(b-2)+2<t_{a}(G)+t_{b}(G).
		\]
		Finally, if $\Hred_{l-1}(\Delta_1\cap \Delta_2; \mathbb{K}) \neq 0$,
		then  by (\ref{HochstersFormula}),
		$\beta_{a+b-2,t_{a+b}(G)-3}^R(R/I(G)) \neq 0$. Thus, $t_{a+b}(G)
		\leq t_{a+b-2}(G)+3$. Thus, by Proposition \ref{indsub},  
		\[
		t_{a+b}(G) \leq
		t_{a}(G)+t_{b-2}(G)+3=t_{a}(G)+2(b-2)+3<t_{a}(G)+t_{b}(G).
		\]
		Hence, the assertion follows.
	\end{proof}
	
	\begin{remark}
		The above two propositions together prove that if one of the indices
		is bounded above by $\nu(G) + 1$, then $t_{a+b}(G) \leq t_a(G) +
		t_b(G)$. Comparing with \cite[Theorem 6.2]{ACI} in the case of Koszul
		monomial ideals, we can see that while they have a possibly bigger
		upper bound on $a$ and $b$, our hypothesis puts restriction only
		on one of $a$ and $b$. While we may not be able to directly achieve
		subadditivity from their result, in our case, we obtain subadditivity.
	\end{remark}
	We have immediate consequence of Propositions \ref{indsub} and
	\ref{nug+1}:
	\begin{corollary}
		If $G$ is a graph such that $\pdim(R/I(G)) \leq 2\nu(G) + 2$, then
		subadditivity holds for $G$.
	\end{corollary}
	It follows from the above corollary and \cite[Corollary
	7.6.30]{JTh} that the subadditivity holds for cycles.
	
	We now introduce a hereditary class of graphs along with some extra
	hypotheses.
	
	\begin{definition}\label{class}
		Let $\G$ be a hereditary class of finite simple graphs with the
		property that  
		$G \in \G$ if and only if
		\begin{enumerate}
			\item $G = C_n$ for some $n \geq 3$ \hspace*{0.5cm} or
			\item for $G \neq C_n$ for any $n \geq 3$, there exists $e =
			\{x, y\} \in E(G)$ such that $N_G(x) \subset N_G[y]$
			and $G\setminus e \in
			\mathcal{G}$.
		\end{enumerate}
	\end{definition}
	\begin{remark}\label{rmk}
		Let $G$ be a disconnected graph. It is easy to observe that $G \in \mathcal{G}$ if and only if each component of $G$ is in $\mathcal{G}$.
	\end{remark}
	It is immediate from the definition that $\G$ contains cycles and
	forests. We now show that $\G$ contains some important classes
	of graphs. A subset $U$ of $V(G)$ is said to be a \textit{clique} if
	$G[U]$ is a complete graph.  A vertex $v$ is said to be a
	\textit{simplicial vertex} if it belongs to exactly one maximal clique
	of $G$.
	\begin{lemma}\label{chordal-tech}
		Let $G$ be a chordal graph. Then  there exist an edge $e=\{x,y\} \in
		E(G)$ such that $N_G(x)\subset N_G[y]$ and
		$G\setminus e$ is a chordal graph.
	\end{lemma}
	\begin{proof}
		It follows from \cite{dirac61} that $G$ has a simplicial  vertex, say
		$x$. Let $y \in N_G(x)$. Then $e=\{x,y\} \in E(G)$ is an edge such
		that $N_G(x) \subset N_G[y]$. We claim that $G\setminus e$ is a
		chordal graph.  Suppose $G \setminus e$ is not a chordal graph.
		Then $G \setminus e$ contains an induced cycle $C$ of length at
		least $4$. Since $G$ is chordal, this implies that $e$ is a chord
		in $G$ connecting two vertices of $C$.  This contradicts the fact
		that $x$ is a simplicial vertex. Therefore, $G \setminus e$ is a
		chordal graph.
	\end{proof}    
	Let $G$ and $H$ be two graphs. The {\it clique-sum} of $G$ and $H$
	along a complete graph $K_m$ is a graph with the vertex set $V (G)
	\cup V (H)$ and edge set $E(G)  \cup  E(H)$ such that the induced
	subgraph on the vertex set $V(G) \cap V(H)$ is the complete graph
	$K_m$.
\begin{lemma}\label{cliquesum}
Let $G$ be a graph obtained  by clique sum of a cycle $C_m$, $m \geq 4$, 
and some chordal graphs. If $G \neq C_m$, then
there exists an edge $e=\{x,y\} \in E(G)$ such that
$N_G(x)\subset N_G[y]$ and 
$G\setminus e$ is a graph obtained by clique sum of $C_m$ and
some chordal graphs. 
\end{lemma}
\begin{proof}
		Suppose $G \neq C_m$. Then $G$
		is a clique sum of $C_m$ and some chordal graphs, say $G_1,
		\ldots, G_t$. By a theorem of Dirac, \cite{dirac61}, each $G_i$ is
		either a clique or contains two simplicial vertices that are
		non-adjacent. In either case, $G$ will contain a simplicial
		vertex, say $x$. Without loss of generality, assume that $x \in
		V(G_1) \setminus V(C_m)$. Let $y \in
		N_{G_1}(x)$ and set $e=\{x,y\}$. Then $N_G(x) = N_{G_1}(x) \subset
		N_{G_1}[y] \subset N_G[y]$. As in the proof of Lemma
		\ref{chordal-tech}, one can see that $G_1 \setminus e$ is a
		chordal graph. Therefore, $G \setminus e$ is a clique sum of $C_m$ and
		the chordal graphs $G_1 \setminus e, G_2, \ldots, G_t$.
	\end{proof}
	
	We now show that $\G$ contains chordal graphs and their clique sum
	with a cycle.
	\begin{theorem}\label{chordal}
		We have the followings:
		\begin{enumerate}
			\item If $G$ is a chordal graph,  then $G \in \mathcal{G}$.
			%\item If $G$ is a unicyclic graph, then $G\in \mathcal{G}$.
			\item If $G$ be a graph obtained by taking clique sum of a cycle,
			$C_m$, and some chordal graphs, then $G \in \mathcal{G}$.
		\end{enumerate}
	\end{theorem}
	\begin{proof}
		\par (1) Let $G$ be a chordal graph. We prove this
		by induction on $|E(G)|.$ If $|E(G)|=1$, then the assertion is
		true. Assume that $|E(G)|>1$ and for every chordal graph $H$ with
		$|E(H)| < |E(G)|$, $H \in \mathcal{G}$. Every induced subgraph of
		a chordal graph is chordal. If $H$ is a proper induced subgraph of
		$G$, then $|E(H)| <|E(G)|$ and $H$ is a chordal graph. Thus, by
		induction, $H \in \mathcal{G}$. By Lemma \ref{chordal-tech}, there
		exist an edge $e =\{x,y\}$ such that  $N_G(x)\subset N_G[y]$ and $G\setminus e$ is a chordal graph.
		Now, by induction, $G\setminus e \in \mathcal{G}$. Hence, $G \in
		\mathcal{G}.$

		\par(2)
		We prove this by induction on $|E(G)| \geq m$. If  $|E(G)|=m$, then
		$G$ is a cycle and hence, $G \in \mathcal{G}$. Assume that
		$|E(G)|>m$.  Let $H$ be a proper induced subgraph of $G$. 
		If $H$ is a  chordal graph, then by (1),
		$H \in \mathcal{G}$. Suppose $H$ is not a chordal
		graph. Then $H$ is a disjoint
		union of a chordal graph and a graph obtained from clique sum of
		$C_m$ and some chordal graphs. By Remark \ref{rmk} and
		the first part, it is enough to prove that if $H$ is a
		clique sum of a cycle and some chordal graphs, then $H\in
		\mathcal{G}$, and this follows by induction, as $H$ is a proper
		induced subgraph of $G$. As $G$ is not $C_m$, by Lemma
		\ref{cliquesum}, there exist an edge $e=\{x,y\} \in E(G)$ such
		that $N_G(x)\subset N_G[y]$ and $G\setminus e$ is a graph obtained by clique sum of $C_m$ and
		some chordal graphs. Since $|E(G\setminus e)|<|E(G)|$, by
		induction $G\setminus e \in \mathcal{G}$. Hence, $G \in
		\mathcal{G}$.
	\end{proof}
For the following result, we consider the ring $R$ to be
$\NN^n$-graded. We require only the graded version of this result in
this and the next section. We require the multigraded version in the
last section while studying multigraded Betti numbers. Set $\deg x_i =
\ee_i$, the standard basis vector with $1$ at the $i$-th place and
zero everywhere else.  For $u \in V(G)$, set $\ee_u = \ee_i$ if $u =
x_i$. 
	\begin{proposition}\label{splitting}
		Let $G$ be a graph on $V(G)$ and  $e=\{x,y\} \in E(G)$. If
		$N_G(x) \subset N_G[y]$, then the mapping
		cone applied to 
		$$	0 \rightarrow \frac{R}{I(G\setminus e) : xy} (-\ee_x-\ee_y)
		\xrightarrow{\cdot xy} \frac{R}{I(G\setminus e)}$$
		gives the minimal free resolution of ${R}/{I(G)}$.
		In particular, for all $i \geq 0$ and $\mathbf{a} \in \mathbb{N}^n$,
		\begin{align*}
			\beta_{i,\mathbf{a}}^R\left(\frac{R}{I(G)}\right)&=
			\beta_{i,\mathbf{a}}^R\left(\frac{R}{I(G\setminus
				e)}\right)+\beta_{i-1,\mathbf{a}}^R\left(\frac{R}{I(G\setminus
				e) : xy}(-\ee_x-\ee_y)\right) \\&
			= \beta_{i,\mathbf{a}}^R\left(\frac{R}{I(G\setminus
				e)}\right)+\beta_{i-1,\mathbf{a}-\ee_x-\ee_y}^R\left(\frac{R}{I(G\setminus
				e) : xy}\right).
		\end{align*} 
	\end{proposition}
	
	\begin{proof}
		We claim that $I(G\setminus e): xy  =I(G\setminus e): y$.
		Note that  $I(G \setminus e) : xy = (I(G \setminus e) : y) : x$.
		Since $N_G(x) \subset N_G[y]$, it follows from \cite[Lemma
		3.1]{DHS13}, that $x$ does not divide any of the minimal monomial
		generators of $I(G \setminus e) : y$. Hence $I(G \setminus e) : xy
		= I(G \setminus e) : y$.
		Consider, the short exact sequence \begin{eqnarray}\label{ses1}
			0 \rightarrow \frac{R}{I(G\setminus e) : xy} (-\ee_x-\ee_y)
			\xrightarrow{\cdot xy} \frac{R}{I(G\setminus e)} \rightarrow \frac{R}{I(G)}
			\rightarrow 0.
		\end{eqnarray} It follows from \cite[Theorem 2.4]{Leila18} that the minimal free resolution of $R/I(G)$ is obtained by mapping cone construction on $	0 \rightarrow [{R}/({I(G\setminus e) : xy})] (-\ee_x-\ee_y)
		\xrightarrow{\cdot xy} {R}/{I(G\setminus e)}$. Hence, the assertion follows.
\end{proof}

\begin{theorem}\label{subedge}
Subadditivity holds for every $G \in \mathcal{G}$. 
\end{theorem}
\begin{proof}
We prove this by induction on the number of edges. Let $G \in
\mathcal{G}$. If $|E(G)|=1$, then the assertion clearly holds. Assume 
that $|E(G)|>1$ and the result is true for all graphs $H \in
\mathcal{G}$ with $|E(H)|<|E(G)|$. If $G$ is a cycle, then the
subadditivity holds for $G$, \cite[Example 1(b)]{BH17}. Assume that
$G$ is not a cycle graph. Then, by Definition \ref{class}(3),  $G$ has an edge
$e=\{x,y\}$ such that $N_G(x) \subset N_G[y]$. Consider
the following short exact sequence: 
\begin{equation}\label{ses}
	0 \rightarrow \frac{R}{I(G\setminus {e}): xy} ({-2})
	\xrightarrow{\cdot xy} \frac{R}{I(G\setminus {e})} \rightarrow \frac{R}{I(G)} \rightarrow 0.\end{equation}
Note that $G\setminus e$ is graph with $|E(G\setminus e)| <|E(G)|$ and
$G\setminus e \in \mathcal{G}$. Therefore by induction, the subadditivity
holds for $G\setminus e$. From the proof of Proposition
\ref{splitting}, we have $I(G\setminus e) :xy =I(G\setminus e):y =I(G\setminus
N_G[y]) + (N_{G\setminus e}(y))$, where the last equality
follows from
\cite[Lemma 3.1]{DHS13}. Now, $H=G\setminus N_G[y]$ is an induced
subgraph of $G$, therefore,  $H \in \mathcal{G}$. Since $|E(H)|<|E(G)|$, by induction, the
subadditivity holds for $H$. As, $(N_{G\setminus
	e}(y))$ is generated by a regular sequence, the subadditivity holds for $( N_{G\setminus e}(y))$, by Corollary \ref{csub}. Since tensor product of the minimal
free resolutions of $R/I(H)$ and $R/( N_{G\setminus e}(y))$
resolves $I(G\setminus e):xy$, the subadditivity holds for
$R/(I(G\setminus e):xy)$, by Theorem \ref{tsub}. Note that $t^R_1(R/(I(G\setminus e):xy)) \leq 2$.
Therefore, by Theorem \ref{subad}, the resolution obtained by the
mapping cone construction applied to the map $[R/(I(G\setminus
e):xy )](-2) \xrightarrow{\cdot xy} R/I(G\setminus e)$ satisfies
the subadditivity condition. Now, by Proposition \ref{splitting}, the
subadditivity holds for $G$.
\end{proof}

\begin{remark}{\em
In the proof of the above theorem, the properties that we used to
derive the subadditivity property of $R/I(G)$ are that the mapping
cone applied to $[R/(I(G\setminus e) : e](-2) \to R/I(G\setminus
e)$ gives the minimal free resolution of $R/I(G)$ and that
subadditivity holds for $R/(I(G \setminus e) : e)$ and $R/I(G
\setminus e)$. Hence, if $G$ has an edge $e$ satisfying these
properties, then $R/I(G)$ has subadditivity property. This way, one
can possibly get more classes of graphs whose edge ideals have
subadditivity property. For example, if we take $G$ to be the graph
$C_6$ along with an edge $\{x_1, x_4\}$ and $e = \{x_2,x_3\}$, then $G
\setminus e$ is a unicyclic graph and $I(G \setminus e) : e$ is a
tree. Hence both these ideals satisfy subadditivity. In this case, it
is easy to verify that the mapping cone gives a minimal free
resolution of $R/I(G)$. Hence $R/I(G)$ satisfies subadditivity. Note
that $G \notin \G$.
}
\end{remark}

We say that a vertex $v \in V(G)$ is a \textit{cut vertex} if $G
\setminus v$ has  more components than $G$.   A \textit{block} of a
graph is a  maximal nontrivial connected induced subgraph which has no
cut vertex.   If exactly one block of a graph $G$ is a cycle and each
other block is an edge, then we say that $G$ is a \textit{unicyclic
graph}.
	
	There is a containment of graph classes:
	\[ \{\text{unicyclic graphs} \} \subset \{ \text{semi-block
		graphs} \}
	\subset  \{\text{clique sum of a cycle and chordal
		graphs}\}.\]
	A unicyclic graph is a clique sum of a cycle and some trees.
	The notion of semi-block graph was introduced in \cite{AR}. We refer
	the readers to \cite[Section 3]{AR} for the definition of semi-block
	graphs. Now, as a consequence of Theorems \ref{chordal} and
	\ref{subedge}, we show that for some important classes of graphs, the
	subadditivity holds. This includes chordal graphs for which the
	subadditivity was proved by Bigdeli and Herzog in \cite{BH17}.
	%\cite[Corollary 3]{BH17} due to Bigdeli and Herzog: 
	\begin{corollary}\label{unicyc}
		If $G$ is a chordal graph, semi-block graph or a unicyclic graph, then
		the subadditivity holds for $G$.
	\end{corollary}
	Let $H$ be a graph and $U \subset V(H)$. \textit{The cone of $H$
	along $U$,}
	denoted by $x*_U H,$ is the graph on the vertex set $V(H)\cup \{x\}$
	and edge set $E(H) \cup \{ \{x,u\} \;:\; u \in U\}$. If $U =
	V(H)$, then we simply write $x * H$.
	\begin{theorem}\label{vertexcover-cone}
		Let $H$ be a non-trivial  graph, and $U\subset V(H)$ be a vertex cover of
		$H$. Let $G=x*_U H$. If subadditivity
		holds for $H$, then it holds for $G$. 
	\end{theorem}
	\begin{proof}
		Note that $I(G) =I(H)+(xu : u \in U)$. It follows from \cite[Corollary
		4.3]{HV07} that $$I(H) \cap (xu : u \in U)= xI(H).$$ Let $H'$ be the
		subgraph of $G$ on the vertex set $\{x\} \cup U$ and edge set
		$\{\{x,u\} : u \in U\}$. It  follows from Theorem \cite[Theorem
		4.6]{HV07} that 
		\begin{eqnarray}\label{betti4}
			\beta_{i,j}^R\left(\frac{R}{I(G)}\right) & = & 
			\beta_{i,j}^R\left(\frac{R}{I(H)}\right)+
			\beta_{i-1,j}^R\left(\frac{R}{xI(H)}\right)+
			\beta_{i,j}^R\left(\frac{R}{I(H')}\right) \nonumber \\
			&= & \beta_{i,j}^R\left(\frac{R}{I(H)}\right)+
			\beta_{i-1,j-1}^R\left(\frac{R}{I(H)}\right)+
			\beta_{i,j}^R\left(\frac{R}{I(H')}\right).
		\end{eqnarray}
		Therefore, for  each $ i \geq 1$, $t_i(G)
		=\max\{t_i(H),t_{i-1}(H)+1,t_i(H')\}$. Note that $H'$ is
		a star graph. Thus, by \cite[Corollary 3]{BH17}, for all $a, b \geq 1$ with $a +b \leq
		\text{pd}(R/I(H'))$, $t_{a+b}(H') \leq t_a(H')+t_b(H')$.
		For all $a,b \geq 1$ with $a+b \leq \text{pd}(R/I(G))$, 
		\begin{align*}
			&t_{a}(G)+t_{b}(G)\\&=
			\max\{t_{a}(H),\; t_{a-1}(H)+1,t_a(H')\}+\max\{t_{b}(H),\; t_{b-1}(H)+1,t_b(H')\}\\&\geq
			\max \{t_{a}(H)+t_{b}(H),\; t_{a-1}(H)+t_{b}(H)+1, t_a(H')+t_b(H')\}\\&\geq
			\max \{t_{a+b}(H),t_{a+b-1}(H)+1,\; t_{a+b}(H')\}\\&=
			t_{a+b}(G).
		\end{align*}
		Hence, the subadditivity holds for $G$.
	\end{proof}
	
	As a consequence of Theorem \ref{vertexcover-cone}, we obtain the
	following result:
\begin{corollary}\label{main-cor}
Let $G$ be a graph on $V(G)$. Then, the subadditivity holds for
$G$ if
\begin{enumerate}
	\item $G = W_n = x * C_n$, the wheel graph on $n+1$ vertices;
	\item $G = J_{2,n}= x *_U C_{2n}$, Jahangir graph on
		$2n+1$ vertices, where $U$ is a vertex cover of
		$C_{2n}$ of size $n$;
	\item $G$ is a complete multipartite graph; 
	\item $G = F_{m,n} = (x_1 *_U (\cdots *_U (x_m *_U P_n)))$, where
		$U = V(P_n)$, is a fan graph. 
\end{enumerate}
\end{corollary}
	
	Note that the classes of the graph listed above are not in $\G$. In
	similar manner, one can keep constructing several graphs $G \notin
	\G$ satisfying subadditivity.
	
	Taking a join of two graphs is an important operation in graph theory.
	For two graphs $G$ and $H$, $G*H$ is the graph on the vertex set
	$V(G*H) = V(G) \sqcup V(H)$ and with the edge set $E(G*H) = E(G) \sqcup
	E(H) \cup \{ \{x,y\} ~ : ~ x \in V(G) \text{ and } y \in V(H)\}$. It
	is natural to ask how the subadditivity of $G$ and $H$ gets translated
	to $G*H$. We answer this question below.
	\begin{theorem}
		Let $G$ and $H$ be graphs on $m$ and $n$ vertices, respectively. If subadditivity holds for $G$ and $H$, then so for $G*H$.
	\end{theorem}
	\begin{proof}
		It follows from \cite[Corollary 3.4]{Amir12} that for all $i ,j$, \begin{align}\label{join-betti}
			\beta_{i,j}^R\left(\frac{R}{I(G*H)}\right)= \sum_{k=0}^{j-2}\Bigg[ {{n}\choose {k}}\beta_{i-k,j-k}^R\left(\frac{R}{I(G)}\right) +{ m \choose k}\beta_{i-k,j-k}^R\left(\frac{R}{I(H)}\right)\Bigg]. 
		\end{align}
		First, we claim that for each $1 \leq i \leq n+m-1$, 
		\begin{align*}
			t_i(G*H)= \max \{\{t_{i-k}(G)+k : 0 \leq k \leq \min\{i,n\}\}\cup\{t_{i-l}(H)+l : 0 \leq l \leq \min\{i,m\}\}\}. 
		\end{align*}
		Set $p=\max \{\{t_{i-k}(G)+k : 0 \leq k \leq
		\min\{i,n\}\}\cup\{t_{i-l}(H)+l : 0 \leq l \leq \min\{i,m\}\}\}$.
		If $q >p$, then for each $0 \leq k \leq \min\{i,n\}$, $t_{i-k}(G)
		<q-k$, and for each $0 \leq l \leq \min\{i,m\}$, $t_{i-l}(H)
		<q-l$. It follows from \eqref{join-betti} that
		$\beta_{i,q}^R(R/I(G*H))=0$. Thus, $t_i(G*H) \leq p$. Now, to
		prove our claim, it is enough to prove that
		$\beta_{i,p}^R(R/I(G*H)) \neq 0$. Note that for some $0 \leq k
		\leq \min\{i,n\}$, $t_{i-k}(G)=p-k$ or 		for some $0 \leq l
		\leq \min\{i,m\}$, $t_{i-l}(H)=p-l$. Hence, using
		\eqref{join-betti}, we get $\beta_{i,p}^R(R/I(G*H)) \neq 0$.
		This proves the claim. 
		
		Now, consider for $a,b \geq 1$, 
		\begin{align*}
			& t_{a+b}(G*H)\\ & =\max \Big\{\{t_{a+b-k}(G)+k : 0 \leq k \leq \min\{a+b,n\}\}\\&\;\;\;\;\;\;\;\;\;\;\;\;\;\;\;\;\;\;\;\;\;\;\;\;\;\;\;\;\cup\{t_{a+b-l}(H)+l : 0 \leq l \leq \min\{a+b,m\}\}\Big\}\\ &
			=\max \Big\{\{t_{a+b-k-k'}(G)+k+k' : 0 \leq k+k' \leq \min\{a+b,n\}\}\\& \;\;\;\;\;\;\;\;\;\;\;\;\;\;\;\;\;\;\;\;\;\;\;
			\;\;\;\;\;\cup\{t_{a+b-l-l'}(H)+l+l' : 0 \leq l+l' \leq \min\{a+b,m\}\}\Big\} \\ &
			\leq \max \Big\{\{t_{a+b-k-k'}(G)+k+k' : 0 \leq k \leq \min\{a,n\},  0 \leq k' \leq \min\{b,n\}\}\\&\;\;\;\;\;\;\;\;\;\;\;
			\;\;\;\;\;\;\;\;\;\;\;\;\;\;\;\;\;\cup\{t_{a+b-l-l'}(H)+l+l' : 0 \leq l \leq \min\{a,m\}, 0 \leq l' \leq \min\{b,m\}\}\Big\} \\&
			\leq \max \Big\{\{t_{a-k}(G)+k +t_{b-k'}(G)+k' : 0 \leq k \leq \min\{a,n\},  0 \leq k' \leq \min\{b,n\}\}\\&\;\;\;\;\;\;\;\;\;\;\;\;
			\;\;\;\;\;\;\;\;\;\cup\{t_{a-l}(H)+l+t_{b-l'}(H)+l' : 0 \leq l \leq \min\{a,m\}, 0 \leq l' \leq \min\{b,m\}\}\Big\} \\&
			\leq \max \Big\{\{t_{a-k}(G)+k : 0 \leq k \leq \min\{a,n\}\}\cup\{t_{a-l}(H)+l : 0 \leq l \leq \min\{a,m\}\}\Big\}\\& 
			\;\;\;\;\;\;+\max \Big\{\{t_{b-k'}(G)+k' : 0 \leq k' \leq \min\{b,n\}\}\cup\{t_{b-l'}(H)+l' : 0 \leq l' \leq \min\{b,m\}\}\Big\}\\&
			= t_a(G*H)+t_b(G*H).
		\end{align*}	
		This completes the proof.
	\end{proof}
	
	%A graph $G$ is said to be \textit{bipartite} if 
	%there is a bipartition of $V(G)=V_1 \sqcup V_2$ such that for each
	%$i=1,2$, $G[V_i]$ is a trivial graph.
	\subsection{Path ideals of rooted tree} Path ideals are generalization
	of edge ideals. In this subsection, we prove that the subadditivity
	for path ideals of rooted trees. 
	
	A tree together with a fixed vertex is called a {\it rooted tree}, and
	the fixed vertex of that tree is called a {\it root}. In
	a tree, there exists a unique path between any two given
	vertices. Thus, we can see that there is a unique
	path between the root and any other vertex in a rooted tree. We can also view a rooted
	tree as a directed graph by
	assigning to each edge the direction that goes ``away'' from the root.
	From now onward, $\Gamma$ denotes a rooted tree with $x$ as root and
	$\Gamma$ is viewed as a directed
	rooted tree in the above sense.   An edge $\{u,v\}$ in a rooted tree
	whose direction is from $u$ to $v$ is denoted by $(u,v)$. Let $\Gamma$
	be a rooted tree on $\{x_1, \ldots, x_n\}$ with root $x=x_i$ for some
	$i$. Let $t \geq 1$.
	A {\it directed path} of length $(t-1)$ in $\Gamma$ is a sequence of
	distinct vertices $x_{i_1},\ldots,x_{i_t}$ such that
	$(x_{i_j},x_{i_{j+1}}) \in E(\Gamma)$ for each $j \in
	\{1,\ldots,t-1\}$. The {\it $t$-path ideal} of $\Gamma$ is denoted by
	$I_t(\Gamma)$ and is defined as $$I_t(\Gamma):= (x_{i_1}\cdots
	x_{i_t}\; : \: x_{i_1},\ldots,x_{i_t} \text{ is a path of length }
	(t-1) \text{ in } \Gamma).$$
	For a vertex $u$ of $\Gamma$, the {\it level} of $u$ is denoted by
	level$(u)$ and is length of the unique path from $x$ to $u$. The {\it
		height} of $\Gamma$, denoted by height$(\Gamma)$, is $\max_{v \in
		V(\Gamma)} \text{level}(v)$. A vertex $u$ is said to be {\it parent}
	of a vertex $v$ if $(u,v) \in E(\Gamma)$, and a vertex $w$ is said to
	be {\it child} of a vertex $v$ if $(v,w) \in E(\Gamma)$. A vertex $v$
	is said to be {\it descendant} of a vertex $u$ if there is a path from
	$u$ to $v$. An induced subtree of $\Gamma$ with root $u \in V(\Gamma)$
	is an induced subtree of $\Gamma$ on the vertex set $\{u\} \cup \{v\;
	: \; v \text{ is descendant of } u \}$. We conclude this section by
	discussing subadditivity of $t$-path ideals of rooted trees:
	\begin{theorem}\label{rooted-tree}
		Let $\Gamma$ be a rooted tree on $\{x_1, \ldots, x_n\}$. Then, for all
		$1 \leq t \leq \text{height}(\Gamma) +1$, subadditivity holds for
		$R/I_t(\Gamma)$.
	\end{theorem}
	\begin{proof}
		We prove this by induction on the number of vertices of $\Gamma$. Let
		$x_{i_t}$ be a vertex such that level$(x_{i_t})=$
		height$(\Gamma)$. Let $x_{i_1},\ldots, x_{i_t}$ be a path of
		length $(t-1)$ terminating at $x_{i_t}$. Consider the following
		short exact sequence
		\begin{align}\label{pathseq}
			0 \rightarrow \frac{R}{I_t(\Gamma \setminus \{x_{i_t}\}):x_{i_1}\cdots x_{i_t}} (-t) \xrightarrow{x_{i_1}\cdots x_{i_t}} \frac{R}{I_t(\Gamma \setminus \{x_{i_t}\})} \rightarrow \frac{R}{I_t(\Gamma)} \rightarrow 0.
		\end{align}
		By induction, the subadditivity holds for $R/I_t(\Gamma \setminus
		\{x_{i_t}\})$.  Let $x_{i_0}$ be the only parent of $x_{i_1}$, if
		it exists.	For $j \in \{0,\ldots,t\}$, let $\Gamma_j$ be the
		induced subtree of $\Gamma$ rooted at $x_{i_j}$, and for  $j \in
		\{0,\ldots,t-1\}$, let $\Delta_j=\Gamma[V(\Gamma_j) \setminus
		V(\Gamma_{j+1})]$. By \cite[Lemma 2.8]{BHO}, $$I_t(\Gamma
		\setminus \{x_{i_t}\}):x_{i_1}\cdots x_{i_t} = I_t(\Gamma
		\setminus \{x_{i_0},\ldots,x_{i_t}\}) + (x_{i_0}) +
		\sum_{j=0}^{t-1}I_{t-j}(\Delta_j \setminus
		\{x_{i_0},\ldots,x_{i_t}\}).$$ Observe that $t_1(R/(I_t(\Gamma
		\setminus \{x_{i_t}\}):x_{i_1}\cdots x_{i_t})) \leq t$.
		It follows from \cite[Remark 2.9]{BHO} that the minimal free
		resolution of $R/(I_t(\Gamma \setminus \{x_{i_t}\}):x_{i_1}\cdots
		x_{i_t})$ is obtained by the tensor product of the minimal free
		resolution of $R/I_t(\Gamma \setminus
		\{x_{i_0},\ldots,x_{i_t}\})$, $ R/(x_{i_0})$ and
		$R/I_{t-j}(\Delta_j \setminus \{x_{i_0},\ldots,x_{i_t}\})$ for $j
		\in \{0,\ldots,t-1\}$. Now, by induction, the subadditivity holds
		for $R/I_t(\Gamma \setminus \{x_{i_0},\ldots,x_{i_t}\})$ and for
		$j \in \{0,\ldots,t-1\}$, $R/I_{t-j}(\Delta_j \setminus
		\{x_{i_0},\ldots,x_{i_t}\})$. Thus, subadditivity holds for
		$R/(I_t(\Gamma \setminus \{x_{i_t}\}):x_{i_1}\cdots x_{i_t})$, by
		Theorem \ref{tsub}.
		It follows from \cite[Theorem 2.7]{BHO} that the mapping cone
		construction applied to \eqref{pathseq} gives a minimal free
		resolution of $R/I_t(\Gamma)$. Hence, by Theorem \ref{subad}, the
		subadditivity holds for $R/I_t(\Gamma)$.
	\end{proof}

\section{Strand connectivity of edge ideals}
In this section, we discuss the strand connectivity of edge
ideals. 
It is known that the tensor product of minimal free resolutions of two
homogeneous ideals on disjoint set of variables gives a minimal free
resolution of their sum. It is interesting to ask if this property
translates to strand connectivity. The following example shows that
this is not the case.
\begin{example}\label{discon-strand}
Let $I = I(C_5) = (x_1x_2, \ldots, x_5x_1) \subset R_1 =
\kk[x_1,\ldots, x_5]$ and $J = (y_1y_2y_3) \subset R_2 
\subset \kk[y_1,y_2,y_3]$. Then $I$ is strand connected,
\cite[Example 1(b)]{BH17} and $J$ being a complete intersection, it is
strand connected. It can be seen that for $I + J \subset
R = \kk[x_1,\ldots, x_5,y_1,y_2,y_3]$, $\beta_{1,3}^R(R/(I+J)) =
1, ~\beta_{2,4}^R(R/(I+J)) = 0$ and $\beta_{3,5}^R(R/(I+J)) = 1$. This
shows that the $2$-strand of $I+J$ is not connected.
\end{example}

However, we show that we can get strand connectivity if one of the
ideals is generated by linear forms.
	
\begin{lemma}\label{tstrand}
Let $I$ and $ J $ be homogeneous ideals of $R$ such that there exist
minimal generating sets for $I$ and $ J $ in disjoint sets of
variables. If $I$ is generated by $k$ linear forms and $ J $ is
strand connected, then  $ I+J $ is strand connected.
\end{lemma}

\begin{proof}
Since the minimal generating sets of $I$ and $J$ are in disjoint
variables,  the tensor product of the minimal free resolution of
$R/I$ and $R/J$ gives the minimal free resolution of $R/(I+J)$.
Since $I$ is generated by $k$ linear forms, for all $i,j$, we have
\begin{eqnarray}\label{bettiprod}
	\beta_{i,i+j}^R\left(\frac{R}{I+J}\right) = \sum_{0 \leq r\leq
		\min\{i,k\}}
	\beta_{r,r}^R\left(\frac{R}{I}\right) 
	\beta_{i-r,i-r+j}^R\left(\frac{R}{J}\right).
\end{eqnarray}
Thus, if $j$-strand of $I+J$ is nonempty, then we claim  that the
$j$-strand of $J$ is nonempty and   $j\text{-strand of } I+J =
[q_j(J),p_j(J)+\min\{k,q_j(J)\}].$ If $j$-strand of $J$ is empty,
then it follows from (\ref{bettiprod}) that $j$-strand of $I+J$ is
empty. Thus, $j$-strand of $J$ is nonempty.  Let $i \in
[q_j(J),p_j(J)+\min\{k,q_j(J)\}]$. If $i \leq
p_j(J)$, then 
$\beta_{i,i+j}^R({R}/{J})\neq 0$, and hence it follows from
(\ref{bettiprod}) that 
$\beta_{i,i+j}^R(R/(I+J)) \neq 0$. If $i>p_j(J)$, then for some $l
\leq \min\{k,q_j(J)\}$, $i-l = p_j(J)$. Therefore,
$\beta_{l,l}^R({R}/{I})\beta_{i-l,i-l+j}^R({R}/{J})\neq 0$, and
hence, $\beta_{i,i+j}^R(R/(I+J)) \neq 0$. Thus, for $i \in
[q_j(J),p_j(J)+\min\{k,q_j(J)\}]$, $\beta_{i,i+j}^R(R/(I+J)) \neq
0$. Now, let $i \notin [q_j(J),p_j(J)+\min\{k,q_j(J)\}]$. If
$i<q_j(J)$, then for each $0 \leq r \leq \min\{k,i\}$, $i-r <
q_j(J)$, and hence, $\beta_{i,i+j}^R(R/(I+J) ) = 0$. If $i >
p_j(J) + \min\{k, q_j(J)\}$, then for any $0 \leq r \leq \min\{i,
k\}$, $i - r > p_j(J)$. Hence so that $\beta_{i-r,i-r+j}^R(R/J) =
0$. Therefore $\beta_{i,i+j}^R(R/I+J) = 0$. This completes the proof.
\end{proof}

In the example above, one of the ideals is generated in degree $3$.
If both ideals are generated in degree $2$ and are strand connected,
then can one say that their sum is strand connected? This question has
a graph theoretic analogue too. If $G$ and $H$ are two disjoint graphs
whose edge ideals are strand connected, then is $I(G \sqcup H)$ strand
connected?

Now we begin the study of strand connectivity of edge ideals of
graphs.
If $j > \reg(R/I(G))$, then the $j$-strand is empty. Therefore, to
study the strand connectivity, we can restrict ourselves to $j \leq
\reg(R/I(G))$. It is known that $\nu(G) \leq \reg(R/I(G)),$
\cite{Kat2006}. We have earlier shown that the graphs in $\G$ satisfy
subadditivity.  
We now consider a subcollection of $\G$.  Let 
\[\G':= \{G \in \G ~ : ~ C_n \text{ is not an induced
	subgraph of G for }n \equiv 2~(\text{mod }3)\}.\]
We first show that graphs in this collection have minimal regularity.
\begin{proposition}\label{greg}
	Let $G \in \mathcal{G}'$. Then, $\reg(R/I(G)) = \nu(G)$.
\end{proposition}
\begin{proof}
The inequality $\nu(G) \leq \reg(R/I(G))$ follows from \cite[Lemma
2.2]{Kat2006}.  We now prove that $\reg(R/I(G)) \leq \nu(G)$.
We proceed by induction on the number of vertices. Let $G \in
\mathcal{G}'$.  If $|V(G)|=2$, then the assertion is true.
Assume that $|V(G)|>2$ and the result is true for all graphs
$H \in \mathcal{G}'$ with $|V(H)|<|V(G)|$. If $G$ is a cycle,
then $|V(G)| \not\equiv 2~(\text{mod } 3)$, and thus, by
\cite[Theorem 7.6.28]{JTh}, $\reg(R/I(G)) = \nu(G)$. Suppose
that $G$ is not a cycle. Then, by Definition \ref{class}(2b),
$G$ has an edge $e=\{x,y\}$ such that $N_G(x) \subset N_G[y]$
and $G\setminus e \in \G$. Since $G\setminus\{y\}$
and $G\setminus N_G[y]$ are induced subgraphs of $G$, by
Definition \ref{class}(2a), $G\setminus\{y\}, \; G \setminus
N_G[y] \in \mathcal{G}'$. Thus, by induction,
$$\reg\left(\frac{R}{I(G\setminus \{y\})}\right) \leq
\nu(G\setminus \{y\})$$ and $$\reg \left(\frac{R}{I(G\setminus
	N_G[y])}\right)\leq \nu(G\setminus N_G[y]).$$ 
Since $G \setminus \{y\}$ is an induced subgraph of $G$, $\nu(G
\setminus \{y\} \leq \nu(G)$. If $\{e_1, \ldots, e_s\}$ is an induced
matching in $G \setminus N_G[y]$, then $\{e_1, \ldots, e_s, \{x,
y\}\}$ is an induced matching in $G$. Therefore, 
$\nu(G\setminus
N_G[y])+1 \leq \nu(G)$. It follows from \cite[Theorem
2.7]{STT15} that $$ \reg \left( \frac{R}{I(G)} \right)  \leq
\max \Big\{\reg\left(\frac{R}{I(G\setminus \{y\})}\right),\reg
\left(\frac{R}{I(G\setminus N_G[y])}\right)+1 \Big\} \leq
\nu(G).$$ Hence, the assertion follows.
	\end{proof}

	\begin{theorem}
		If $G \in \mathcal{G}'$, then  $I(G)$ is strand connected.
	\end{theorem}
	\begin{proof}
		We prove this by induction on the number of edges. Let $G \in
		\mathcal{G}'$. If $|E(G)|=1$, then the assertion is true.  Assume
		that $|E(G)|>1$ and the result is true for all graphs $H \in
		\mathcal{G}'$ with $|E(H)|<|E(G)|$. If $G$ is a cycle, then by
		\cite[Example 1(b)]{BH17}, $j$-strand of $I(G)$ is connected.
		Assume that $G$ is not a cycle.  Then, by Definition
		\ref{class}(2b), $G$ has an edge $e=\{x,y\}$ such that $N_G(x)
		\subset N_G[y].$ Since $N_G(x) \subset N_G[y]$, the edge $e$ can not 
		be a chord of an induced cycle of length at least $5$. Therefore, $G
		\setminus e$ does not contain $C_n$ as an induced cycle for any $n
		\equiv 2~(\text{mod }3)$ so that 
		$G\setminus e \in \mathcal{G}'$. 
		Since by Proposition \ref{greg} $\reg(R/I(G)) = \nu(G)$ for all $G \in \G'$, $j$-strand is empty
		for all $j >  \nu(G)$.  Let $ 1 \leq j \leq \nu(G)$. 
		It follows from \cite[Lemma 2.2]{Kat2006} that $q_j(I(G))=j$. Suppose
		that for some $i>j$, $\beta_{i,i+j}^R(R/I(G))=0$. It is enough to show
		that $\beta_{i+1,i+1+j}^R(R/I(G)) =0$.
		By Proposition \ref{splitting}, we have 
		\begin{eqnarray}\label{betti-ind}
			\beta_{i,i+j}^R\left(\frac{R}{I(G)}\right)=
			\beta_{i,i+j}^R\left(\frac{R}{I(G\setminus e)}\right) +
			\beta_{i-1,i+j-2}^R\left(\frac{R}{(I(G\setminus e):e)}\right).
		\end{eqnarray}
		Since
		$\beta_{i,i+j}^R(R/I(G))=0$, we have
		$\beta_{i,i+j}^R(R/I(G\setminus e))=0$ and
		$\beta_{i-1,i-1+j-1}^R(R/(I(G\setminus e):e))=0$. By induction,
		$I(G\setminus e)$ is strand connected. Moreover, it follows from
		(\ref{betti-ind}) that $\nu(G \setminus e) \leq \nu(G)$. Hence, 
		$\beta_{i+1,i+1+j}^R(R/I(G\setminus e))=0$ as $i >j$. This implies
		that
		$\beta_{i+1,i+1+j}^R(R/I(G)) =\beta_{i,i+j-1}^R(R/(I(G\setminus
		e):e)).$  Since $G \setminus N_G[y]$ is an
		induced subgraph of $G$, $G \setminus N_G[y] \in \G'$. Thus, by
		induction, $I(G \setminus N_G[y])$ is strand connected.   Note
		that $I(G \setminus e) : e = I(G \setminus
		N_G[y]) + (N_{G\setminus e}[y])$. Therefore, by Lemma \ref{tstrand},
		$I(G\setminus e):e$ is strand connected. Since $i-1 > j-1$ and
		$\beta_{i-1,i-1+j-1}^R(R/(I(G\setminus e):e))=0$, we have
		$\beta_{i,i+j-1}^R(R/(I(G\setminus e):e))=0$. Hence from
		(\ref{betti-ind}), we get $\beta_{i+1,i+1+j}^R(R/I(G)) =0$. 
		This implies that the $j$-strand of $I(G)$ is connected. Hence, $I(G)$
		is strand connected.
	\end{proof}
The above result gives several important classes of graphs whose edge
ideals are strand connected. Before we go on  listing them, we prove
two more results in this direction which enable us to construct more
classes of graphs whose edge ideals are strand connected.

\begin{theorem}\label{strandvertex}
Let $H$ be a non-trivial graph on $n$ vertices, $U$ be a vertex cover of
$H$ and $x$ be a new vertex. Let $G = x *_U H$. 
If $j$-strand of $I(H)$ is connected, then $j$-strand of $I(G)$ is
connected. In particular, if $I(H)$ is strand connected, then
$I(G)$ is strand connected. 
\end{theorem}
\begin{proof}
Let $j>1$ be such that $j$-strand of $I(H)$ is connected.  Let $H'$ be
the subgraph of $G$ on the vertex set $\{x\} \cup U$ and edge set
$\{\{x,u\} : u \in U\}$. 
%It  follows from \cite[Theorem
%4.6]{HV07} that 
%\begin{align*}
%\beta_{i,i+j}^R\left(\frac{R}{I(G)}\right)&=\beta_{i,i+j}^R\left(\frac{R}{I(H)}\right)+\beta_{i-1,i+j}^R\left(\frac{R}{xI(H)}\right)+\beta_{i,i+j}^R\left(\frac{R}{I(H')}\right)\\
		%&=\beta_{i,i+j}^R\left(\frac{R}{I(H)}\right)+\beta_{i-1,i-1+j}^R\left(\frac{R}{I(H)}\right)+\beta_{i,i+j}^R\left(\frac{R}{I(H')}\right).
		%\end{align*}
		Since $H' = K_{1, |U|}$, by \cite[Theorem 5.2.4]{JTh},
		$\beta_{i,i+j}(R/I(H')) = 0$ if $j > 1$. Therefore, we get from
		(\ref{betti4})
		\begin{eqnarray}\label{betag}
			\beta_{i,i+j}^R\left(\frac{R}{I(G)}\right) & = &
			\beta_{i,i+j}^R\left(\frac{R}{I(H)}\right)+
			\beta_{i-1,i-1+j}^R\left(\frac{R}{I(H)}\right).
		\end{eqnarray}
		This implies that for each $1 \leq j \leq \reg(R/I(G))$,
		$q_j(I(G))=q_j(I(H))$.
		%For $i <q_j(I(H))$, we have
		%$$\beta_{i,i+j}^R\left(\frac{R}{I(G)}\right)  = 
		%\beta_{i,i+j}^R\left(\frac{R}{I(H)}\right)+
		%\beta_{i-1,i-1+j}^R\left(\frac{R}{I(H)}\right)=0,$$
		%    and for $i=q_j(I(H))$,
		%    $$\beta_{i,i+j}^R\left(\frac{R}{I(G)}\right)=\beta_{i,i+j}^R\left(\frac{R}{I(H)}\right)+\beta_{i-1,i-1+j}^R\left(\frac{R}{I(H)}\right)=\beta_{i,i+j}^R\left(\frac{R}{I(H)}\right)\neq 0.$$ Thus, for each $1 \leq j \leq \reg(R/I(G))$, $q_j(I(G))=q_j(I(H))$. 
		%Let if possible, the $j$-strand of $I(G)$
		%is not connected. Then, 
		Assume that for some $i > q_j(I(G)) = q_j(I(H))$,
		$\beta_{i,i+j}^R(R/I(G))=0$. It is enough to prove that $\beta_{i+1,i+1+j}^R(R/I(G))=0$. It follows from Equation (\ref{betag}) that
		$\beta_{i,i+j}^R(R/I(H))=0$. Since $j$-strand of $I(H)$ is connected, this
		implies that $\beta_{i+1,i+1+j}^R(R/I(H))=0$. Hence, from Equation
		(\ref{betag}), we get 
		$\beta_{i+1,i+1+j}^R(R/I(G)) =0.$ This shows that
		$j$-strand of $I(G)$ is connected.
	\end{proof}
	\begin{theorem}\label{strandjoin}
		Let $G$ and $H$ be graphs on $m$ and $n$ vertices  such that $I(G)$
		and $I(H)$ are strand connected. Then, $I(G*H)$ is strand
		connected.
	\end{theorem}
	
	\begin{proof}
		Since the linear strand of a homogeneous ideal is connected, we assume that $j>1$.
		It follows from \cite[Corollary 3.4]{Amir12} that for all $i ,j$, \begin{align}\label{join-strand}
			\beta_{i,i+j}^R\left(\frac{R}{I(G*H)}\right)= \sum_{k=0}^{i+j-2}\Bigg[ {{n}\choose {k}}\beta_{i-k,i-k+j}^R\left(\frac{R}{I(G)}\right) +{ m \choose k}\beta_{i-k,i-k+j}^R\left(\frac{R}{I(H)}\right)\Bigg]. 
		\end{align}
		First, we claim that for each $1 < j \leq \reg(R/I(G*H))$,
		$$q_j(I(G*H))=\min \{q_j(I(G)),q_j(I(H))\}.$$ If $i < \min
		\{q_j(I(G)),q_j(I(H))\}$, then by Equation \eqref{join-strand},
		$ \beta_{i,i+j}^R(R/I(G*H))=0,$
		as $i-k < \min \{q_j(I(G)),q_j(I(H))\}$, for each $ 0 \leq k \leq
		i+j-2$. If $i = \min \{q_j(I(G)),q_j(I(H))\}$, then either
		$\beta_{i,i+j}^R(R/I(G)) \neq 0$ or $\beta_{i,i+j}^R(R/I(H)) \neq
		0$. Hence it follows from Equation \eqref{join-strand} that 
		$\beta_{i,i+j}^R(R/I(G*H))\neq 0.$ 
		Thus, for each $1 < j \leq \reg(R/I(G*H))$, $q_j(I(G*H))=\min
		\{q_j(I(G)),q_j(I(H))\}.$ 
		
		Let $q_j(I(G*H)) < i$.
		%have the following cases: \\[1ex]
		%\textbf{Case 1.}
		%If $ i \leq \min\{p_j(I(G)),p_j(I(H))\}$, then either $i \in
		%(q_j(I(G)),p_j(I(G))]$ or $i \in (q_j(I(H)),p_j(I(H))]$.
		%Consequently, either $\beta_{i,i+j}^R(R/I(G)) \neq 0$ or
		%$\beta_{i,i+j}^R(R/I(H)) \neq 0$. Thus, by Equation
		%\eqref{join-strand}, $\beta_{i,i+j}^R(R/I(G*H)) \neq 0$. \\[1ex]
		%\textbf{Case 2.} Suppose that $\min\{p_j(I(G)),p_j(I(H))\} <  i
		%\leq \max\{p_j(I(G)),p_j(I(H))\}$.  Without loss of generality, we
		%may assume that $p_j(I(G)) \leq p_j(I(H))$. If $i \in
		%[q_j(I(H)),p_j(I(H))]$, then $\beta_{i,i+j}^R(R/I(H)) \neq 0$, and
		%hence, $\beta_{i,i+j}^R(R/I(G*H)) \neq 0$. Assume that $p_j(I(G))
		%< i < q_j(I(H))$. For some $0<l<i$, $i-l=p_j(I(G))$. Also, we have $l
		%< i < q_j(I(H)) < n$ so that ${n \choose l} \neq 0$. Hence
		%${n \choose l} \beta_{i-l,i-l+j}^R(R/I(G)) \neq 0$. 
		%Thus, by Equation \ref{join-strand}, $\beta_{i,i+j}^R(R/I(G*H))
		%\neq0$. \\[1ex]
		%\textbf{Case 3.}
		%Suppose that $ \max\{p_j(I(G)),p_j(I(H))\} <i$. 
To prove that the $j$-strand of $I(G*H)$ is connected, it is enough to
prove that if
$\beta_{i,i+j}^R(R/I(G*H))=0$, then $\beta_{i+1,i+1+j}^R(R/I(G*H))=0$.
Suppose $\beta_{i,i+j}^R(R/I(G*H))=0$. We claim that $i > \max\{
q_j(I(G)), q_j(I(H))\}$. We may assume that $\max\{q_j(I(G)),
q_j(I(H)) \} = q_j(I(H))$. First, note that $i$ can not be equal 
either to $q_j(I(G))$ or to $q_j(I(H))$. Assume, if possible, that $i
< q_j(I(H))$.  Let $k > 0$  be such that $i - k = q_j(I(G))$.
Then we have $0 < k < i < n$ so that ${n \choose k}
\beta_{i-k, i-k+j}^R(R/I(G)) \neq 0$.  Therefore, it follows
from Equation \eqref{join-strand} that $\beta_{i,i+j}^R(R/I(G*H)) \neq
0$ which is a contradiction to our assumption. Hence $i > \max\{
	q_j(I(G)), q_j(I(H))\}$.

From Equation \eqref{join-strand}, for 
$0 \leq k \leq i+j-2$, we get 
\begin{eqnarray}\label{eqn5}
{n \choose k}\beta_{i-k,i-k+j}^R\left(\frac{R}{I(G)}\right) = 0 = 
{m \choose k}\beta_{i-k,i-k+j}^R\left(\frac{R}{I(H)}\right).
\end{eqnarray} 
We need to
prove that ${n \choose
k}\beta_{i+1-k,i+1-k+j}^R({R}/{I(G)}) = 0 = {m \choose
k}\beta_{i+1-k,i+1-k+j}^R({R}/{I(H)})$, for $0 \leq k \leq i+j-1$, i.e., we need
to prove that for $-1 \leq k' \leq i+j-2$, ${n \choose
k'+1}\beta_{i-k',i-k'+j}^R({R}/{I(G)}) = 0 = {m \choose
k'+1}\beta_{i-k',i-k'+j}^R({R}/{I(H)})$. For $0 \leq k'
\leq i+j-2$, this follows from Equation \eqref{eqn5}. 
And the case $k' = -1$ follows from the strand connectivity of
$I(G)$ and $I(H)$. This completes the proof.
\end{proof}
%%%%%%%%%%%%%%%%%%%%%%%%%%%%%%%%%%%%%%%%%%
		\iffalse
		Note that for some $0<l<i$,
		$i-l= \max\{p_j(I(G)),p_j(I(H))\}$. Therefore, either
		$\beta_{i-l,i-l+j}^R(R/I(G)) \neq 0$ or
		$\beta_{i-l,i-l+j}^R(R/I(H)) \neq 0$. 
		Also, if $k < l$, then $i - k >
		\max\{p_j(I(G)), p_j(I(H))\}$ so that $\beta_{i-k, i-k+j}(R/I(G))
		= 0 = \beta_{i-k, i-k+j}(R/I(H))$. Therefore, we get from Equation
		\eqref{join-strand}  that 
		\begin{align*}
			0 = \beta_{i,i+j}^R\left(\frac{R}{I(G*H)}\right)=
			\sum_{k=l}^{i+j-2}\Bigg[ {{n}\choose
				{k}}\beta_{i-k,i-k+j}^R\left(\frac{R}{I(G)}\right) +{ m \choose
				k}\beta_{i-k,i-k+j}^R\left(\frac{R}{I(H)}\right)\Bigg]. 
		\end{align*} 
		
		$0 \leq k \leq l$, $\beta_{i+1-k,
			i+1-k+j}(R/I(G)) = 0 = \beta_{i+1-k, i+1-k+j}(R/I(H))$. We need not
		bother about $k > \max\{m, n\}$. Suppose $l+1 \leq k \leq \max\{m,
		n\}$. 
		
		If $k \leq \min\{m, n\}$, then $\beta_{i-k, i-k+j}(R/I(G)) = 0 =
		\beta_{i-k, i-k+j}(R/I(H))$.
		
		Then, wither $m < l \leq n$ and
		$\beta_{i-k,i-k+j}^R(R/I(G))=0$, for $l \leq k \leq n$ or  $n< l
		\leq m$ and $\beta_{i-k,i-k+j}^R(R/I(H))=0$, for $l \leq k \leq
		m$.  Now, using Equation \eqref{join-strand}, we get
		\begin{align*}
			&\beta_{i+1,i+1+j}^R\left(\frac{R}{I(G*H)}\right)\\&= \sum_{k=l+1}^{i+j-1}\Bigg[ {{n}\choose {k}}\beta_{i+1-k,i+1-k+j}^R\left(\frac{R}{I(G)}\right) +{ m \choose k}\beta_{i+1-k,i+1-k+j}^R\left(\frac{R}{I(H)}\right)\Bigg]\\&=0. 
		\end{align*} Hence, the assertion follows. 
		\fi
%%%%%%%%%%%%%%%%%%%%%%%%%%%%%%%%%%%%%%%%%%
	
As a consequence of our results, we obtain several classes of graphs
whose edge ideals are strand connected.
	\begin{corollary}\label{strand-cor}
		Let $G$ be a graph and $I(G)$ be its edge ideals. Then $I(G)$ is
		strand connected if
		\begin{enumerate}
			\item $G$ is chordal graph; (\cite[Proposition 5]{BH17})
			\item $G = W_n = x * C_n$; 
			\item $G = J_{2,n}$, Jahangir graph on $2n+1$ vertices;
			\item $G = F_{m,n}$, the fan graph;
			\item $G$ is a unicyclic graph with the induced cycle of length $n
			\neq 3k+2$ for some $k \geq 1$.
		\end{enumerate}
	\end{corollary}
	
	Other than the named classes of graphs listed above, one can construct
	more graphs using Theorems \ref{strandvertex}, \ref{strandjoin}. It is
	known that not all edge ideals are strand connected. We expect that
	characterizing strand connected edge ideals will be a tough problem. 
	There are even possibly simpler questions in this direction for which
	the answers are unknown:
	
	\begin{question}
		\begin{enumerate}
			\item If $I(G)$ is strand connected, then is $I(H)$ strand
			connected for every non-trivial induced subgraph $H$ of
			$G$?
			%\item Suppose $I$ and $J$ are homogeneous ideals on disjoint
			%variables. If $I$ and $J$ are strands connected, then is
			%$I+J$ strand connected?
			
			\item If $\reg(R/I(G)) = \nu(G)$, then is $I(G)$ strand
			connected?
			
			\item If $j \leq \nu(G)$, is the $j$-strand of $I(G)$ connected?
		\end{enumerate}
	\end{question}
	
	\section{Multigraded Betti numbers of edge ideals of graphs}
	
In this section, we study multigraded Betti numbers of some 
classes of edge ideals. First we generalize the result of Boocher et
al. to the case of unicyclic graphs. Girth of a unicyclic graph $G$ is
the length of the induced cycle of $G$.
\begin{theorem}\label{multi-betti-uc}
Let $G$ be a unicyclic graph on $n$-vertices with girth $m$. 
\begin{enumerate}
\item If $m$ is not a multiple of $3$, then
$\beta_{i,\mathbf{a}}(R/I(G)) \in \{0,1\}$ for all $i \geq 1$
and  $\mathbf{a} \in \mathbb{N}^n$.
\item If $m = 3k$, then $\beta_{i,\mathbf{a}}(R/I(G)) \in \{0,1,2\}$
for all $i \geq 1$ and  $\mathbf{a} \in \mathbb{N}^n $.
Furthermore assume that every vertex in $G$ is at a distance
at most two from the unique cycle in $G$. Then
$\beta_{i,\mathbf{a}}(R/I(G)) =2$ if and only if $i=2k$, and
$\mathbf{a}=\sum_{x \in V(C_m)} \ee_x$.
\end{enumerate}
\end{theorem}
\begin{proof}
		We prove this by induction on $n-m$. If $n-m=0$, then 
		the assertion follows from \cite[Proposition 3.5]{BDGMS19}.
		Now, assume that $n-m>0$. Then, there exists $x
		\in V(G)$  such that $\deg_G(x) =1$. Let $N_G(x) = \{y\}$. 
		%Consider the
		%following multigraded short exact sequence: 
		%     \begin{equation}\label{ses1}
		%0 \rightarrow [I(G\setminus \{x\}):xy] (\mathbf{-e_x-e_y}) \rightarrow I(G\setminus \{x\}) \rightarrow I(G) \rightarrow 0.\end{equation} 
		Then, it follows from Proposition \ref{splitting} that for $i \geq 1$ and
		$\mathbf{a} \in \mathbb{N}^n$, 
		\begin{equation}\label{betti-eq}
			\beta_{i,\mathbf{a}}\left(\frac{R}{I(G)}\right)=\beta_{i,\mathbf{a}}\left(\frac{R}{I(G\setminus
				\{x\})}\right)+\beta_{i-1,\mathbf{a}-\ee_x-\ee_y}\left(\frac{R}{I(G\setminus
				\{x\}):xy}\right).\end{equation} 
		If
		$\mathbf{a}_x\neq 0$, then for any $j \geq 1$, $\beta_{j,\mathbf{a}}(R/I(G\setminus \{x\}))
		=0$ as $x$ does not divide any of the minimal monomial generators of
		$I(G\setminus \{x\})$. If $\mathbf{a}_x =0$, then
		$[\mathbf{a}-\ee_x-\ee_y]_x$ is
		negative. Consequently, for any $j \geq 1$, we have
		$\beta_{j,\mathbf{a}-\ee_x-\ee_y}
		(R/(I(G\setminus \{x\}):xy)) =0$. Thus  only one term on the
		right hand side of Equation (\ref{betti-eq})
		will contribute to $\beta_{i, \mathbf{a}}(R/I(G))$.
		Observe that $G\setminus
		\{x\}$ is a unicyclic graph on $(n-1)$-vertices having girth $m$
		and that $$I(G\setminus \{x\}) :xy
		=I(G \setminus N_G[y])+(N_{G \setminus \{x\}}(y)).$$ Since Koszul
		complex is the minimal free resolution of  $(N_{G \setminus \{x\}}(y))$,
		$\beta_{i,\mathbf{a}}(R/N_{G \setminus \{x\}}(y)) \in \{0,1\}$, for all $i \geq 1$ and $\mathbf{a} \in \mathbb{N}^n$. Observe that $G\setminus
		N_G[y]$ is either a forest or a unicyclic graph. If
		$G\setminus N_G[y]$ is a forest, then by \cite[Theorem
		2.2.2]{BRR}, $\beta_{i,\mathbf{a}}(R/I(G\setminus N_G[y]))
		\in \{0,1\}$ for all $ i \geq 1$ and $\mathbf{a} \in
		\mathbb{N}^n$. 
		
		\par (1) Assume that $m$ is not a multiple of $3$. Since $G \setminus
		\{x\}$ is a unicyclic graph on $n-1$ vertices and having girth $m$, by
		induction, we have $\beta_{i,\mathbf{a}}(R/I(G\setminus \{x\}))
		\in \{0,1\}$ for all $ i \geq 1$ and $\mathbf{a} \in
		\mathbb{N}^n$. If $G\setminus N_G[y]$ is a unicyclic graph,
		then also we may conclude by induction that $\beta_{i,\mathbf{a}}(R/I(G\setminus
		N_G[y])) \in \{0,1\}$ for all $ i \geq 0$ and $\mathbf{a}
		\in \mathbb{N}^n$. Since the generators of $(N_{G \setminus
			\{x\}}(y))$ and $I(G\setminus N_G[y])$ are in disjoint variables,
		tensor product of the minimal free resolutions of $R/(N_{G
			\setminus \{x\}}(y))$ and $R/I(G\setminus N_G[y])$
		gives the minimal free resolution of $R/(I(G\setminus
		\{x\}):xy)$. Thus,
		$\beta_{i,\mathbf{a}}(R/(I(G\setminus \{x\}):xy)) \in
		\{0,1\}$ for all $i \geq 1$ and $\mathbf{a} \in \mathbb{N}^n$. 
		Hence, it follows from (\ref{betti-eq}) that for each $i \geq 1$ and 
		$\mathbf{a} \in \mathbb{N}^n$, $\beta_{i,\mathbf{a}}(R/I(G)) \in 
		\{0,1\}$. 
		\par (2) Assume that $m$ is a multiple of $3$. Observe that $G\setminus \{x\}$ is a
		unicyclic graph on $n-1$ vertices with girth $m$. By induction, we have
		$\beta_{i,\mathbf{a}}(R/I(G\setminus \{x\})) \in \{0,1,2\}$ for
		all $ i \geq 1$ and $\mathbf{a} \in \mathbb{N}^n$.   If
		$G\setminus N_G[y]$ is a unicyclic graph, then by induction,
		$\beta_{i,\mathbf{a}}(R/I(G\setminus N_G[y])) \in \{0,1,2\}$
		for all $ i \geq 1$ and $\mathbf{a} \in \mathbb{N}^n$.
		Thus,
		$\beta_{i,\mathbf{a}}(R/(I(G\setminus \{x\}):xy)) \in
		\{0,1,2\}$ for all $i \geq 1$ and $\mathbf{a} \in \mathbb{N}^n$.
		Hence, it follows from (\ref{betti-eq}) that for each $i \geq 1$ and 
		$\mathbf{a} \in \mathbb{N}^n$, $\beta_{i,\mathbf{a}}(R/I(G)) \in 
		\{0,1,2\}$.
		
		Now, we prove second part. Since $x$ is  at a distance of at most  $2$
		from the unique cycle, $y$ is at a distance at most one from the
		unique cycle, and thus,  $G\setminus N_G[y]$ is a forest. Thus by 
		eq. \eqref{betti-eq}, $\beta_{i,\mathbf{a}}(R/I(G))=2$ if and only if $\beta_{i,\mathbf{a}}(R/I(G\setminus \{x\}))=2$. Since $G \setminus
		\{x\}$ is a unicyclic graph on $n-1$ vertices and having girth $m$. 
		Therefore, by induction, we have $\beta_{i,\mathbf{a}}(I(G\setminus
		\{x\}))=2$ if and only if $i=2k$ and $\mathbf{a}= \sum_{x\in
			V(C_m)}\ee_x$. Hence, the assertion follows.
	\end{proof}

In the previous two sections, we saw that the knowledge of 
Betti numbers of a graph would give information about 
Betti numbers of the cone of that graph along a vertex cover. 
We prove a similar result for
multigraded Betti numbers here.
	
\begin{theorem}\label{mvc}
Let $H$ be a non-trivial graph on $n$-vertices. Let $U$ be a vertex
cover of $H$ and $x$ be a new vertex. Let $G = x *_U H$. If 
	
\begin{minipage}{\linewidth}
\begin{minipage}{0.43\linewidth}
$$\begin{array}{cc}
\beta_{i,\mathbf{a}}\left(\frac{R}{I(H)}\right) \leq	\Big\{ \begin{array}{cc}
c &	\text{ if } \;|\mathbf{a}|=i+1, \\
d & \text{ if } \; |\mathbf{a}|>i+1,
\end{array}
\end{array}$$
\end{minipage}
\begin{minipage}{0.05\linewidth}
then
\end{minipage}
\begin{minipage}{0.42\linewidth}
$$\begin{array}{cc}
\beta_{i,\mathbf{a}}\left(\frac{R}{I(G)}\right) \leq	\Big\{ \begin{array}{cc}
c+1 &	\text{ if } \;|\mathbf{a}|=i+1, \\
d & \text{ if } \; |\mathbf{a}|>i+1.
\end{array}
\end{array}$$
\end{minipage}
\end{minipage}
\end{theorem}

\begin{proof}
Let $H'$ be the
subgraph of $G$ on the vertex set $\{x\} \cup U$ and edge set
$\{\{x,u\} : u \in U\}$. It  follows from \cite[Theorem
4.6]{HV07} that 
\begin{align*}
	\beta_{i,\mathbf{a}}^{R}\left(\frac{R}{I(G)}\right)&=\beta_{i,\mathbf{a}}^R\left(\frac{R}{I(H)}\right)+\beta_{i-1,\mathbf{a}}^R\left(\frac{R}{xI(H)}\right)+\beta_{i,\mathbf{a}}^R\left(\frac{R}{I(H')}\right)\\
	&=\beta_{i,\mathbf{a}}^R\left(\frac{R}{I(H)}\right)+\beta_{i-1,\mathbf{a}-\ee_x}^R\left(\frac{R}{I(H)}\right)+\beta_{i,\mathbf{a}}^R\left(\frac{R}{I(H')}\right).
\end{align*}
Let $ i\geq 1$ and let $\mathbf{a} \in \mathbb{N}^{n+1}$ such that $|\mathbf{a}| = i+1$. Since $H'$ is a tree, by \cite[Theorem 2.2.2]{BRR}, $\beta_{i,\mathbf{a}}^R(R/I(H')) \in \{0,1\}$. 
If
$\mathbf{a}_x\neq 0$, then for any $i \geq 1$, $\beta_{i,\mathbf{a}}^R(R/I(H))
=0$ as $x$ does not divide any of the minimal monomial generators of
$I(H)$. If $\mathbf{a}_x =0$, then $[\mathbf{a}-\ee_x]_x = -1$.
Consequently, for any $i \geq 1$, we have $\beta_{i,\mathbf{a}-\ee_x}^R
(R/I(H)) =0$, and $\beta_{i,\mathbf{a}}^R(R/I(H')) =0$. Thus,  $\beta_{i,\mathbf{a}}^R(R/I(G)) \leq c+1$.

Now, assume that $\mathbf{a} \in \mathbb{N}^{n+1}$ such that $|\mathbf{a}|>i+1$. Since $H' =
K_{1, |U|}$, by \cite[Theorem 5.2.4]{JTh}, $\beta_{i,\mathbf{a}}^R(R/I(H')) = 0$ as $|\mathbf{a}| > i+1$. Therefore, 
\begin{eqnarray}\label{betag1}
	\beta_{i,\mathbf{a}}^R\left(\frac{R}{I(G)}\right) & = &
	\beta_{i,\mathbf{a}}^R\left(\frac{R}{I(H)}\right)+
	\beta_{i-1,\mathbf{a}-\ee_x}^R\left(\frac{R}{I(H)}\right).
\end{eqnarray}
If
$\mathbf{a}_x\neq 0$, then for any $i \geq 1$, $\beta_{i,\mathbf{a}}^R(R/I(H))
=0$ as $x$ does not divide any of the minimal monomial generators of
$I(H)$. If $\mathbf{a}_x =0$, then $[\mathbf{a}-\ee_x]_x = -1$.
Consequently, for any $i \geq 1$, we have $\beta_{i,\mathbf{a}-\ee_x}^R
(R/I(H)) =0$. Thus  only one term in Equation \eqref{betag1} will
contribute to $\beta_{i, \mathbf{a}}^R(R/I(G))$, and hence,
$\beta_{i,\mathbf{a}}^R(R/I(G)) \leq d$.
\end{proof}
As a consequence, we obtain upper bounds for the multigraded Betti
numbers of several classes of graphs.

\begin{corollary}\label{multi-betti-jh}
	Let $U$ be a vertex cover of $C_n$, $x$ be a vertex and 
	$G=x*_U C_n $.  Then $\beta_{i,\mathbf{a}}^R(R/I(G))
	\leq 2$ for all $i \geq 1$ and $\mathbf{a} \in \mathbb{N}^{n+1}$. In
	particular, if $G = W_n$, the wheel graph on $n+1$ vertices or $G =
	J_{2,n}$, the Jahangir graph, then $\beta_{i,\mathbf{a}}^R(R/I(G))
	\leq 2$ for all $i \geq 1$ and $\mathbf{a} \in \mathbb{N}^{|V(G)|}$.
\end{corollary}
\begin{proof}
Let $ i \geq 1$ and $\mathbf{a} \in \mathbb{N}^{n+1}$. It follows from
\cite[Proposition 3.5]{BDGMS19} that $\beta_{i,\mathbf{a}}^R(R/I(C_n))
\leq 2$. If $|\mathbf{a}|>i+1$, then, by Theorem
\ref{mvc}, $\beta_{i,\mathbf{a}}^R(R/I(G)) \leq 2.$ Now, assume
that $|\mathbf{a}|=i+1$. Then, following the notation and proof
of Theorem \ref{mvc}, we get 
$$\begin{array}{cc}
\beta_{i,\mathbf{a}}\left(\frac{R}{I(G)}\right) =\Bigg\{ \begin{array}{cc}
\beta_{i, \mathbf{a}}^R\left(\frac{R}{I(C_n)}\right) &	\text{ if }
\;\mathbf{a}_x=0, \\
\beta_{i-1, \mathbf{a}-\ee_x}^R\left(\frac{R}{I(C_n)}\right)+ \beta_{i,\mathbf{a}}^R\left(\frac{R}{I(H')}\right) & \text{ if } \; \mathbf{a}_x \neq 0.
\end{array} 
\end{array}$$ Since $H'$ is a tree, by \cite[Theorem 2.2.2]{BRR},
$\beta_{i,\mathbf{a}}^R(R/I(H')) \in \{0,1\}$. Now
the first part of the Corollary follows from 
\cite[Proposition 3.5]{BDGMS19}. The second part follows by observing 
that $W_n = x * C_n$ and $J_{2,n} = x *_U C_{2n}$, where 
$U = \{x_{2k-1} ~ : ~ 1 \leq k \leq n\}$. 
\end{proof}
	
\begin{corollary}\label{multi-betti-fan}
Let $G =F_{m,n}$ be the fan graph on $n+m$ vertices, $n \geq 2, m
\geq 1$. Then $\beta_{i,\mathbf{a}}^R(R/I(G)) \leq 2$ for all $i \geq
1$ and $\mathbf{a} \in \NN^{n+m}$.
\end{corollary}
\begin{proof}
We do this by induction on $m$. If $m = 1$, then the assertion is
immediate from Theorem \ref{mvc} since $F_{1,n} = x_1 * P_n$.
Assume that $\beta_{i, \mathbf{a}}^R(R/I(F_{m-1,n})) \leq 2$ for all
$i \geq 1$ and $\mathbf{a} \in \NN^{n+m}$. If $|\mathbf{a}| >
i+1$, then it follows from Theorem \ref{mvc} that
$\beta_{i,\mathbf{a}}(R/I(F_{m,n})) \leq 2$. If $|\mathbf{a}| = i+1$,
then following the proof of Theorem \ref{mvc}, we get
$$\begin{array}{cc}
\beta_{i,\mathbf{a}}\left(\frac{R}{I(F_{m,n})}\right) =
\Bigg\{ \begin{array}{cc}
\beta_{i, \mathbf{a}}^R\left(\frac{R}{I(F_{m-1,n})}\right) &	\text{ if }
	\;\mathbf{a}_{x_m}=0, \\
	\beta_{i-1, \mathbf{a}-\ee_{x_m}}^R\left(\frac{R}{I(F_{m-1,n})}\right)+
	\beta_{i,\mathbf{a}}^R\left(\frac{R}{I(H')}\right) & \text{ if }
	\; \mathbf{a}_{x_m} \neq 0.
\end{array} 
\end{array}$$ 
If $\aaa_{x_i} \neq 0$ for some $1 \leq i \leq m-1$, then
$\beta_{i,\aaa}^R(R/I(H')) = 0$ so that the assertion holds true. If
$\aaa_{x_i} = 0$ for all $1 \leq i \leq m-1$ and $\aaa_{x_m} \neq 0$,
then $\beta_{i,\aaa}^R(R/I(H')) \leq 1$.  In this case 
$[\aaa - \ee_{x_m}]_{x_j} = 0$ for all $1 \leq j \leq m$
so that $\beta_{i,
\aaa - \ee_{x_m}}^R(R/I(F_{m-1},n)) = \beta_{i,
\aaa - \ee_{x_m}}^R(R/I(P_n))\leq 1$. This completes the proof.
\end{proof}

\begin{corollary}\label{multi-betti-cm}
Let $G$ be a complete $k$-partite graph on $n$-vertices. Then,
$\beta_{i,\mathbf{a}}^R(R/I(G)) \leq k-1$, for all $i \geq 1$ and
$\mathbf{a} \in \mathbb{N}^{n}$.
\end{corollary}
\begin{proof}
Let $ i \geq 1$ and $ \mathbf{a} \in \mathbb{N}^n$. If
$|\mathbf{a}|>i+1$, then $\beta_{i,\mathbf{a}}^R(R/I(G))=0$,
\cite{JTh}.
Assume that $|\mathbf{a}|=i+1$. We prove this by induction on $k$.
Assume that $k=2$. Then, $G$ is a complete bipartite graph.
Therefore, $V(G)=V_1 \sqcup V_2$ such that $G[V_i]$ is a trivial
graph for $i=1,2$ and $G = G[V_1]*G[V_2]$. 
We now proceed by induction on $|V_2|$. If
$|V_2|=1$, then $G$ is a tree and thus, by \cite[Theorem
2.2.2]{BRR}, the result holds. Let $V_2 =\{x_1,\ldots,x_r\}$ with
$r >1$. Set $V_2'=\{x_1,\ldots,x_{r-1}\}$. Note that
$G'=G[V_1]*G[V_2']$ is a complete bipartite graph. By induction,
$\beta_{i,\mathbf{a}}^R(R/I(G')) \leq 1$, for all $i$ and $\mathbf{a}$.
Since $V_1$ is a vertex cover of $G'$, it follows from the proof
of  Theorem \ref{mvc} that
$$\begin{array}{cc}
\beta_{i,\mathbf{a}}\left(\frac{R}{I(G)}\right) =\Bigg\{ \begin{array}{cc}
\beta_{i, \mathbf{a}}^R\left(\frac{R}{I(G')}\right) &	\text{ if } \;\mathbf{a}_{x_r}=0, \\
\beta_{i-1, \mathbf{a}-\ee_{x_r}}^R\left(\frac{R}{I(G')}\right)+ \beta_{i,\mathbf{a}}^R\left(\frac{R}{I(H')}\right) & \text{ if } \; \mathbf{a}_{x_r} \neq 0.
\end{array} 
\end{array}$$ 
If $\mathbf{a}_{x_j} \neq 0$ for some $1 \leq j \leq r-1$, then
$\beta_{i, \mathbf{a}}(R/I(H')) = 0$. If $\aaa_{x_j} = 0$ for $1
\leq j \leq r-1$ and $\aaa_{x_r} \neq 0$,
then $\beta_{i-1, \mathbf{a}-\ee_{x_r}}^R(R/I(G')) =
\beta_{i-1, \mathbf{a}-\ee_{x_r}}^R(R/I(G[V_1])) =
0.$ Hence the assertion follows for the case $k = 2$.

Assume that $k > 2$. Let $V(G) = V_1 \sqcup \cdots \sqcup V_k$ and
that the result holds true for any complete $(k-1)$-partite graph. Now we prove
the assertion for $k$ by induction on $|V_k|$. Let $G' = G[V_1 \sqcup \cdots \sqcup V_{k-1}]$. Suppose $V_k =
\{x\}$. Then $G =
	x* G'$. Hence the result follows from Theorem \ref{mvc}. Let $V_k =
\{x_1, \ldots, x_r\}$, $r > 1$. Take $U = V_1 \sqcup \cdots \sqcup
V_{k-1}$ and observe that $G = x_r *_U G[U
\sqcup \{x_1, \ldots x_{r-1}\}]$. Now arguments similar to the proof of the case $k = 2$, we
get the required assertion.
\end{proof}
	
We conclude the article with a question on multigraded Betti numbers:
\begin{question}
Given upper bounds for multigraded Betti numbers of $I(G)$ and $I(H)$,
can one obtain an upper bound for multigraded Betti numbers of $I(G*H)$?
%More precisely, if $\beta_{i,\mathbf{a}}^R(R/I(G)) \leq r$ and
	%$\beta_{i,\mathbf{b}}^R(R/I(H)) \leq s$ for all $i \geq 0,
	%\mathbf{a} \in \NN^{|V(G)|}, \mathbf{b} \in \NN^{|V(H)|}$ and
	%for some $r, s \in \NN$, then
	%can one obtain a function on $r, s$, say $f(r,s)$, such that
	%$\beta_{i,\mathbf{a}}^R(R/I(G*H)) \leq f(r,s)$ for all $i \geq 0$,
	%$a \in \NN^{|V(G)|+|V(H)|}$?
\end{question}
	%\nocite*{}
	\bibliographystyle{plain}  %% or 
	\bibliography{bilbo}

\end{document}